\documentclass[11pt,a4paper]{article}
\usepackage[utf8]{inputenc}
\usepackage{amsmath}
\usepackage{amsfonts}
\usepackage[colorlinks,linkcolor=blue,anchorcolor=blue,citecolor=blue]{hyperref}
\usepackage{amssymb}
\usepackage{extarrows}
\usepackage{amsthm}
\usepackage{hyperref}
\usepackage{cite}
\usepackage{enumerate}
\PassOptionsToPackage{normalem}{ulem}
\usepackage{ulem}
\usepackage[margin=1in]{geometry} 
\usepackage{array}
\usepackage{cases}
\usepackage{mathrsfs}
\usepackage{longtable}
\allowdisplaybreaks[4]

\newcolumntype{L}[1]{>{\raggedright\let\newline\\\arraybackslash\hspace{0pt}}m{#1}}
\newcolumntype{C}[1]{>{\centering\let\newline\\\arraybackslash\hspace{0pt}}m{#1}}
\newcolumntype{R}[1]{>{\raggedleft\let\newline\\\arraybackslash\hspace{0pt}}m{#1}}

\usepackage{natbib}
\usepackage{color}
\usepackage{dsfont}
\usepackage{marginnote}
\usepackage{enumitem}
\usepackage{verbatim}
\usepackage{graphicx}
\theoremstyle{plain}
\newtheorem{theorem}{\protect Theorem}[section]
\newtheorem{prop}[theorem]{\protect Proposition}
\newtheorem{definition}[theorem]{\protect Definition}

\newtheorem{lemma}[theorem]{\protect Lemma}
\newtheorem{remark}[theorem]{\protect Remark}
\newtheorem{ass}{\protect Assumption}
\newtheorem{corollary}[theorem]{\protect Corollary}

\def\d{\mathrm{d}}

\def\as{\mathrm{a.s.}}

\newcommand{\R}{\mathbb{R}}
\newcommand{\E}{\mathbb{E}}
\newcommand{\X}{\mathcal{X}}
\newcommand{\T}{\top}
\newcommand{\F}{\mathcal{F}}
\newcommand{\Fb}{\mathbb{F}}
\newcommand{\Pb}{\mathbb{P}}
\newcommand{\Pc}{\mathcal{P}}

\newcommand{\tr}{\mathrm{tr}}
\newcommand{\Law}{\mathcal{L}}
\newcommand{\st}{\mathrm{s.t.}}

\newcommand{\pa}{\partial}
\newcommand{\M}{\mathbb{M}}
\newcommand{\Mc}{\mathcal{M}}

\newcommand{\U}{\mathscr{U}}

\newcommand{\Lb}{\mathbb{L}}
\newcommand{\lv}{\lVert}
\newcommand{\rv}{\rVert}
\newcommand{\co}{\mathrm{co}}
\newcommand{\cl}[1]{\overline{#1}}
\newcommand{\Wc}{\mathcal{W}}
\newcommand{\tX}{\mathtt{X}}
\newcommand{\tY}{\mathtt{Y}}
\newcommand{\tH}{\mathtt{H}}

\newcommand{\assref}[1]{\hyperref[#1]{Assumption \ref*{#1}}}
\newcommand{\thmref}[1]{\hyperref[#1]{Theorem \ref*{#1}}}
\newcommand{\propref}[1]{\hyperref[#1]{Proposition \ref*{#1}}}
\newcommand{\remref}[1]{\hyperref[#1]{Remark \ref*{#1}}}
\newcommand{\lemref}[1]{\hyperref[#1]{Lemma \ref*{#1}}}
\newcommand{\defref}[1]{\hyperref[#1]{Definition \ref*{#1}}}
\newcommand{\corref}[1]{\hyperref[#1]{Corollary \ref*{#1}}}
\newcommand{\equref}[1]{\hyperref[#1]{(\ref*{#1})}}
\newcommand{\exaref}[1]{\hyperref[#1]{Example \ref*{#1}}}

\title{Extended mean-field control under constraints: The generalized Fritz-John conditions and Lagrangian method}

\author{Lijun Bo \thanks{Email: lijunbo@ustc.edu.cn, School of Mathematics and Statistics, Xidian University, Xi'an, 710126, China.}
\and
Jingfei Wang \thanks{Email:wjf2104296@mail.ustc.edu.cn, School of Mathematical Sciences, University of Science and Technology of China, Hefei, 230026, China.}
\and
Xiang Yu \thanks{Email: xiang.yu@polyu.edu.hk, Department of Applied Mathematics, The Hong Kong Polytechnic University, Kowloon, Hong Kong, China.}
}
\date{\vspace{-0.3cm}}

\begin{document}
\maketitle
\begin{abstract}
	This paper studies mean-field control with joint law dependence under dynamic expectation constraints and/or dynamic state-control-law constraints. We pioneer the establishment of the stochastic maximum principle (SMP) and the derivation of the backward SDE (BSDE) from the perspective of constrained optimization using the method of Lagrangian multipliers. We first propose to embed the constrained mean-field control (C-MFC) with joint-law dependence into some abstract optimization problems with constraints on Banach spaces, for which we develop the generalized Fritz-John (FJ) optimality conditions. We then prove the stochastic maximum principle (SMP) for C-MFC by transforming the FJ conditions into an equivalent stochastic first-order condition associated with a general type of constrained forward-backward SDEs (FBSDEs). Contrary to the existing literature, we treat the  McKean-Vlasov SDE as an infinite-dimensional equality constraint such that the BSDE induced by the FJ first-order optimality conditions can be interpreted as the generalized Lagrange multiplier. We also employ the methodology to stochastic control and mean field game problems under dynamic constraints.
	
\end{abstract}

\vspace{0.1in}
\section{Introduction}\label{sec:intro}

Motivated by wide applications in finance, engineering and management science, stochastic control problems with various types of constraints have attracted tremendous attention over the past decades. Several types of constraints on state process, the control process and their laws have been predominantly investigated in different context such as the dynamic pathwise state, control and law constraints, the dynamic expectation constraints on state and/or control paths, and their simplified versions as static constraints only at the terminal time. New techniques are inevitably needed to address these constraints in order to employ the dynamic programming approach or the Pontryagin maximum principle approach.

To name a few in the direction of studies to cope with the pathwise state-control-law constraints, Hu and Zhou \cite{HuZhou2005} investigate a stochastic LQ control problem with random coefficients by employing Tanaka's formula, where the control variable is constrained in a cone; Bonnans and Silva \cite{BS2012} examine the stochastic control problems with control constraints and several equality and inequality constraints over the final state, for which the first order and second order optimality conditions are established; Liu et al. \cite{Liu2021} consider stochastic control with an end-point constraint by establishing a general FJ condition in Banach space where only the control variable is reformulated as an optimization problem; Hu et al. \cite{Huaap22} investigate a stochastic LQ control problem with regime switching, random coefficients and cone control constraint, where the optimal feedback control and optimal cost value are characterized explicitly via two systems of extended stochastic Riccati equations; Zhang and Zhang \cite{ZZ23} study a stochastic LQ control problem with a static equality constraint on the terminal state by using the Lagrangian duality method.

On the other hand, fruitful studies to incorporate dynamic or static expectation constraints can also be found in the literature. For example, Soner and Touzi \cite{ST2002} study a stochastic target problem by using the geometric dynamic programming approach; Bouchard and Nutz \cite{BN2012} develop the weak dynamic programming principle 
for stochastic control problems with expectation constraints that implies the Hamilton-Jacobi-Bellman equation in the viscosity sense; Bokanowski et al. \cite{BPZ2016} employ the level-set approach to reformulate the constrained control problems into an optimization problem over a family of unconstrained singular control problems; Chow et al. \cite{ChowYZ2020} reformulate some state constraints as equality expectation constraints and examine the dynamic programming principle; Pfeiffer et al. \cite{PTZ2021} develop a duality approach to study the stochastic control problems under both equality and inequality expectation constraints using convex analysis techniques; Hu et al. \cite{Hu} consider optimal control of stochastic differential equations subject to dynamic expectation constraint where the constrained FBSDEs are formulated and studied; Bayraktar and Yao \cite{BYao24a} recently analyze an optimal stopping problem with some inequality and equality expectation constraints in a general non-Markovian framework and establish a dynamic programming principle in weak formulation; Bayraktar and Yao \cite{BYao24b} further generalize the methodology in \cite{BYao24a} to address stochastic control/stopping problems where the diffusion can be controlled. 

MFC problems, also called the McKean-Vlasov optimal control problems, feature the cooperative interactions in the mean field model where all agents jointly optimize the social optimum that yields the optimal control from the perspective of the social planner, which have raised a lot of research interests in the past few years. To name a few, Acciaio et al.~\cite{Carmona1} explore suitable versions of the Pontryagin stochastic maximum principle (SMP) for extended MFC problems in which the cost function and the state process depend on the joint distribution of the state and the control process; The first-order variation in Bensoussan \cite{Bensoussan1981} and the spike-variation method in Peng \cite{Peng1990} have been generalized to cope with MFC problems among \cite{AD2010,Buckdahn2011, MB2012, Djehicheetal2015} that handle McKean-Valsoc dynamics with coefficients depending on the moments of the population state distribution; Li~\cite{Li2012} examines MFC problems and establish the SMP in scalar interaction forms; Bensoussan et al.~\cite{Bensoussan2013} introduce a unified approach for MFC and MFG problems by using the HJB-FP coupled equations and the SMP; Carmona and Delarue \cite{Carmona2015} solve a class of MFC problems with a general dependence on the law of the state process; Pham and Wei \cite{PW17} investigate the dynamic programming principle for MFC problems in the presence of common noise; Pham and Wei \cite{PW18} examine the extended MFC with state-control joint law dependence by reformulating the problem into a deterministic control problem with only the marginal distribution of the process as controlled state variable; Bo et al. \cite{BLY} consider a LQ-MFC problem by proving the Gamma convergence of the optimal controls from the {$N$-player cooperative games} to the mean-field model {as $N$ tends to infinity}; Nie and Yan~\cite{Nie2022} study extended MFC problems under partial information; Bo et al. \cite{BWWX24} recently establish the SMP and HJB eqaution for an extended MFC problem with Poissonian common noise by proposing the extension transformation.

MFC and MFG problems under the dynamic expectation constraints and/or the dynamic state-control and law constraints are relatively under developed. Grammatico et al.~\cite{Grammatico16} study a deterministic MFC problem in the presence of heterogeneous convex constraints at the level of individual agents; Du and Wu \cite{DW2022} consider a LQG MFC problem with control constraints; Germain et al. \cite{Germain23} apply the level-set approach to some MFC problems under dynamic state-law constraints and discuss some applications in renewable energy storage and portfolio selections; Bonnans et al. \cite{BGP2023} analyze the existence of equilibria for a class of deterministic MFGs of controls with state and control constraints at the terminal time by using Kahutani's theorem; Daudin \cite{Daudin23a} study the convergence problem of mean-field control theory in the presence of state constraints; Daudin \cite{Daudin23b} examine a problem of optimal control of the Fokker-Planck equation with state constraints in the Wasserstein space and give first-order necessary conditions for optimality in the form of a mean-field game system of partial differential equations associated with an exclusion condition;
Craig et al.~\cite{Craig24} discuss a class of first-order MFC problem  with source and terminal constraints by using a particle method; Meherrem and Hafayed \cite{MHafa2024} prove Peng's necessary optimality conditions for a general mean-field system under state constraints based on conventional variational principle.

In the present paper, we study extended mean field control problems under both the dynamic expectation constraints and the pathwise state-control and law constraints from the perspective of optimization problems with Banach space valued constraints. We propose a new methodology that bridges the SMP method for the C-MFC problems and the Lagrange multipliers method based on the generalized Fritz-John (FJ) optimality condition. {More importantly, due to the complex  dynamical constraints (as shown by \eqref{expected_constraint} and \eqref{path_constraint}) inherent in our MFC problem \eqref{cost_func}, the variational formulation of the state process $X^{\alpha}=(X_t^{\alpha})_{t\in[0,T]}$ w.r.t. the control variable $\alpha=(\alpha_t)_{t\in[0,T]}$ becomes intractable when only the control is treated as the target variable (namely, the state process $X^{\alpha}$ is endogenous). To overcome this challenge, we treat both the state process and
the control process as target variables (i.e.,  $(X,\alpha)\in C$ with $C$ being a closed and convex subset of a Banach space) and propose the abstract optimization framework \eqref{optimization} where $F(X,\alpha)=0$ determined by the the dynamics of $X^{\alpha}$ is formulated as an equality constraint.} In a nutshell, when the forward McKean-Vlasov state process is taken as an infinite-dimensional constraints, the generalized FJ optimality condition for the control variable $\alpha$ yields the stochastic (first-order) minimum condition that leads to the SMP. Meanwhile, the generalized Lagrangian multiplier in the space of It\^{o} processes induced by the FJ optimality condition on the state variable $X$ gives us the adjoint process as the solution to the BSDE. 

More concretely, we first develop the general type of FJ optimality condition (\thmref{fj_condition}), new to the literature, for nonlinear optimization problems with (Banach space valued) equality and/or inequality constraints by invoking the generalized gradient introduced in Clarke \cite{Clarke}. Next, as the main novelty of this paper, we formulate our C-MFC problem and transform it into an abstract optimization problem with constraints on Banach spaces (see problem \eqref{optimization}), for which we can apply the newly developed FJ optimality condition on Banach spaces. We then contribute to the establishment of SMP for the C-MFC problem (\thmref{SMP_con}) by transforming the FJ conditions into a stochastic first-order condition that gives rise to a generalized type of constrained FBSDEs. In sharp contrast to the conventional methods, we broaden the scope by considering both state and control as targeted variables in the formulation of the constrained optimization problems such that the Lagrangian multipliers method can be employed. Furthermore, we propose to treat the controlled Mckean-Vlasov state dynamics as an (integral) equality constraint on an infinite dimensional Hilbert space, which is also new to the literature. Then, the adjoint process as a solution to a generalized type of  constrained BSDEs induced by the FJ first-order optimality condition can be interpreted as the generalized Lagrangian multiplier to handle the forward SDE equality constraint, for which  a new space denoted by $\M^2$ is introduced, which consists of all It\^o processes to characterize the integration formulation of the Mckean-Vlasov type SDE constraint. As special cases of our C-MFC problems, we also present the SMP results for the stochastic optimal control problems as well as mean field game (MFG) problem with the same types of dynamic constraints. 

The rest of the paper is organized as follows. Section~\ref{sec:FJcond} gives an overview of the generalized gradients on Banach space and establishes the generalized FJ optimality condition and the constraint qualification (CQ) for a class of constrained optimization problems on Banach spaces. Section~\ref{sec:CSCP} reformulates our C-MFC as an abstract constrained optimization problem on Banach spaces and develops the SMP and the associated constrained FBSDEs by using the generalized FJ optimality condition. Section \ref{sec:constrainedSCP-MFG} presents the SMP results for constrained stochastic optimal control problems and constrained MFG problems as special cases of the main results from Section~\ref{sec:CSCP}. Finally, the proofs of some auxiliary results in are reported in Appendix \ref{appendix}.\\
\noindent{\bf Notations.}\quad We list below some notations that will be used frequently throughout the paper:
\vspace{-0.2in}
\begin{center}
\begin{longtable}{l l}
	$\langle\cdot,\cdot\rangle$ & The inner product in Euclidean space\\
	$|\cdot|$ & Euclidean norm on $\R^n$\\
	$(E^*,\lv\cdot\rv_{E^*})$& The dual space of a Banach space $(E,\lv\cdot\rv_E)$\\
	$B(x;r)$& The ball centered at $x\in E$ with radius $r>0$\\
	$\langle\cdot,\cdot\rangle_{E,E^*}$& The pairing between Banach space $(E,\lv\cdot\rv_E)$ and its dual\\
	$A^c$& The complement of the set $A$\\
	$\co(A)$& The convex hull of the set $A\subset E$\\
	$A^{\rm o}$ ($\cl{A}$)& The interior (closure) of the set $A\subset E$\\
	$\pa A$ & The boundary of the set $A\subset E$\\
	$a\geq_K 0$~($a>_K0$)& $a\in K$~($a\in K^{\rm o}$) with $K$ being a closed (linear) cone\\
	$Df$& The Fr\'{e}chet derivative of $f:E\mapsto\R$\\
	$\nabla_xf$ & Gradient of $f:\R^n\mapsto\R$\\
	$D_Xf$ ($D_{\alpha}f$) & The partial Fr\'{e}chet derivative of $f$ with resepct to $X$ ($\alpha$)\\
	$\mathcal{L}(E_1;E_2)$& Set of continuous linear mappings $f:E_1\mapsto E_2$\\
	$\mathcal{L}(\xi)$ & Law of r.v. $\xi$\\
	$C(E_1;E_2)$ & Set of continuous mappings $f:E_1\mapsto E_2$\\
	$L^p((A,\mathscr{B}(A),\lambda_A);E)$ & Set of $L^p$-integrable $E$-valued mapping defined on \\ 
	& measure space $(A,\mathscr{B}(A),\lambda_A)$. Write $L^p(A; E)$ for short\\ 
	$L^p((A,\mathscr{B}(A),\lambda_A);E)/\mathord{\sim}$& The quotient set of $L^p(A;E)$ w.r.t. equivalence\\
	&relation $\mathord{\sim}$: $f\mathord{\sim} g$ iff $f=g$ except for a $\lambda_A$-null set.  \\ 
	&Write $L^p(A; E)/\mathord{\sim}$ for short\\
	$\E$ ($\E'$) & Expectation operator under probability measure $\Pb$ ($\Pb'$)\\
	$\Pc_p(E)$ & Set of probability measures on E with finite $p$-order moments\\
	$\Mc^+[0,T]$ & Set of non-negative Radon measures on $[0,T]$
\end{longtable}
\end{center}

\section{Generalized Fritz-John Optimality Condition}\label{sec:FJcond}

For given $m,k\in\mathbb{N}$, denote $I:=\{1,2,\ldots,m\}$ and $J:=\{1,2,\ldots,k\}$ as two index sets. Let $(\tX,\|\cdot\|_{\tX})$ be a Banach space, $(\tY_i,\|\cdot\|_{\tY_i})_{i\in I}$ be a sequence of separable Banach spaces and $(\tH_j,\langle\cdot,\cdot\rangle_{\tH_j})_{j\in J}$ be a sequence of Hilbert spaces. 

In this section, we first develop the optimality condition in the Fritz-John (FJ) type for the following constrained optimization problem on Banach spaces that, for a closed and convex subset $C\subset \tX$,
\begin{align}\label{banach_space_optimization}
\begin{cases}
\text{minimize $f(x)$ over $x\in C$;}\\
\text{subject to}~g_i(x)\leq_{K_i} 0,\quad \forall i\in I,\\
\qquad\qquad~~ h_j(x)=0,\quad\forall j\in J,
\end{cases}
\end{align}
where, $K_i\subset\tY_i$ are closed and convex {\it cones} for $i\in I$. The mappings $f:\tX\to\R$, $g_i:\tX\mapsto\tY_i$, $h_j:\tX\to \tH_j$ for $i\in I$ and $j\in J$ are assumed to be continuously Fr\'{e}chet differentiable at any $x\in\tX$. 

For the goal above, we propose a technique based on generalized gradients introduced in Clarke \cite{Clarke} in which the constraints are only restricted in Euclidean space in the sense that $\tY=\tH=\R$ and $K_i=\R_+\cup\{0\}$ for all $i\in I$.  
The general Banach space valued constraints formulated in problem \eqref{banach_space_optimization} are customized to study the C-MFC problems that will be discussed in Section~\ref{sec:CSCP}.

\subsection{A review of generalized gradients on Banach space}
For the completeness, let us give the brief overview of the generalized gradient for a locally Lipschitz mapping $f:\tX\mapsto\R$ given in Clarke \cite{Clarke}. For any $v\in\tX$, we define the generalized directional derivative $f^{\rm o}(x;v)$ in the direction $v$ by
{
\begin{align}\label{generalized_derivative}
f^{\rm o}(x;v):=\varlimsup_{(h,\lambda)\to 0}\frac{f(x+h+\lambda v)-f(x+h)}{\lambda}
\end{align}with $h\in\mathtt{X}$ and $\lambda>0$}.
It follows from Lemma 1 in \cite{Clarke} that $v\mapsto f^{\rm o}(x;v)$ is  convex and $|f^{\rm o}(x;v)|\leq c_f\lv v\rv_{\tX}$ with $c_f>0$ being the locally Lipschitz coefficient of $f:\tX\mapsto\R$. Thus, the generalized gradient of $f:\tX\mapsto\R$ at point $x\in\tX$ (denoted by $\pa f(x)$) is defined to be the subdifferential of the convex function $v\mapsto f^{\rm o}(x;v)$ at $0$. Namely, for $\xi\in\tX^*$, we have $\xi\in\pa f(x)$ iff $f^{\rm o}(x;v)\geq \langle v,\xi\rangle_{\tX,\tX^*}$ for all $v\in\tX$. Therefore, one can conclude that $\pa f(x)$ coincides with the subdifferential of $f$ or $\{Df(x)\}$ when $f$ is convex or has a continuous Fr\'{e}chet derivative $Df:\tX\to\tX^*$, respectively (c.f. Proposition 3 and Proposition 4 in \cite{Clarke}).

Next, we list some basic properties of the generalized gradient whose proofs can be found in \cite{Clarke}:
\begin{lemma}\label{prop_ger_dif} The set $\pa f(x)$ satisfies the following properties:
\begin{itemize}
\item[{\rm(i)}] $\pa f(x)$ is a nonempty, convex and weak-$^*$ compact subset of $\tX^*$. Moreover, it holds that $\lv\xi\rv_{\tX^*}\leq c_f$ for all $\xi\in\pa f(x)$.

\item[{\rm(ii)}] For any $v\in\tX$, it holds that $f^{\rm o}(x;v)=\max_{\xi\in\pa f(x)}\langle v,\xi\rangle_{\tX,\tX^*}$.

\item [{\rm(iii)}] Let $B$ be a nonempty, convex, and weak-$^*$ compact subset of $\tX^*$. Then, $\pa f(x)\subset B$ iff $f^{\rm o}(x;v)\leq \max_{\xi\in B}\langle v,\zeta\rangle_{\tX,\tX^*}$ for all $v\in\tX$.

\item[{\rm(iv)}] If $x\in\tX$ is a local minimum of $f:\tX\mapsto\R$, then $0\in\pa f(x)$.

\item[{\rm(v)}] Let $\xi_i\in\pa f(x_i)$ with $i\in\mathbb{N}$ and assume that $x_i\to x$ and $\xi_i\xrightarrow{w^*}\xi$ as $i\to\infty$. Then,  $\xi\in\pa f(x)$.

\item[{\rm(vi)}] $\pa(f+g)(x)\subset\pa f(x)+\pa g(x)$ for all $x\in\tX$.
\end{itemize}
\end{lemma}

Using Lemma~\ref{prop_ger_dif}, we have the next result.
\begin{lemma}\label{maxdif}
Let $\mathcal{I}$ be an index set and $(\phi_i)_{i\in\mathcal{I}}$ be a family of uniformly local Lipschitz {\rm(}real-valued{\rm)} functions on $\tX$ {\rm(}i.e., their Lipschitz constants are independent of $i\in\mathcal{I}${\rm)}. Define the mapping $\phi:\tX\mapsto\R$ by $\phi(x)=\sup_{i\in\mathcal{I}}\phi_i(x)$ for $x\in\tX$, and assume that $M(x):=\{i\in\mathcal{I};~\phi_i(x)=\phi(x)\}\neq\varnothing$.
Then, $x\mapsto\phi(x)$ is locally Lipschitz and it holds that $\pa\phi(x)\subset\cl{\co\{\pa \phi_i(x);~i\in M(x)\}}.$
\end{lemma}


\begin{remark}\label{remark:Ifinite}
When the index set $\mathcal{I}$ is finite, the right-hand side of the second assertion in \lemref{maxdif} can be replaced by $\co\{\pa\phi_i(x);~i\in M(x)\}$, where $M(x)=\{i\in\mathcal{I};~\phi_i(x)=\phi(x)\}$ is nonempty (c.f. Proposition 9 in Clarke \cite{Clarke}).
\end{remark}

Next, we introduce the distance between a point $x\in C$ and the set $C$ by $d_C(x):=\inf_{c\in C}\lv x-c\rv_{\tX}$. It is easy to see that  $x\mapsto d_C(x)$ is 1-Lipschitz, and hence the generalized gradient $\pa d_C(x)$ is well-defined. Because of the convexity of $C$, for any $x\in C$, the normal cone $N_C(x)$ to the set $C$ at the point $x$ can be defined by 
\begin{align}\label{eq:NCX}
N_C(x):=\{\xi\in\tX^*;~\langle x-c,\xi\rangle_{\tX,\tX^*}\geq 0,~\forall c\in C\}. 
\end{align}
It then follows from Clarke \cite{Clarke} that a point $x\in\tX$ minimizes $f$ over $C$ iff  there exists a neighborhood $V(x)$ of $x$ such that $x$ minimizes $f+c_fd_C$ over $V(x)$. Next, for a cone $K_i\subset \tY_i$ with $i\in I$, let us define
\begin{align}\label{eq:K+}
K_i^{+}:=\left\{\xi\in\tY_i^*;~\langle y,\xi\rangle_{\tY_i,\tY_i^*}\geq 0,~\forall y\in K_i\right\},\quad \forall i\in I.    
\end{align}
Then, it is not difficult to show the following dual relationship between $K_i\subset\tY_i$ and $K_i^+\subset\tY_i^*$:
\begin{lemma}\label{dual}
Let $K_i\subset\tY_i$ be a closed and convex cone for $i\in I$. It holds that
\begin{align}\label{KL}
K_i=\left\{y\in\tY_i;~\langle y,\xi\rangle_{\tY_i,\tY_i^*}\geq 0,~\forall \xi\in K_i^+\right\},\quad\forall i\in I.  
\end{align}
\end{lemma}

We next characterize the interior $K_i^{\rm o}$ of the closed and convex cone $K_i$ in the next lemma, {whose proof is reported in Appendix \ref{appendix}.}
\begin{lemma}\label{interior}
A point $y\in K_i^{\rm o}$ iff $\langle y,\xi\rangle_{\tY_i,\tY_i^*}>0$ for all nonzero $\xi\in K_i^+$.
\end{lemma}

\subsection{The generalized FJ optimality condition on Banach spaces}

In this subsection, we establish both necessary and sufficient optimality conditions for problem \eqref{banach_space_optimization} with Banach space valued constraints, namely the generalized FJ optimality condition.

We first state the generalized FJ necessary optimality condition for problem \eqref{banach_space_optimization}.
\begin{theorem}[FJ necessary condition]\label{fj_condition}
Let $\widehat{x}\in C$ solve problem \eqref{banach_space_optimization} locally. Then, there exist some $r_0\in\R_+$, $\lambda_i\in\tY_i^*$ for $i\in I$, and $\mu_j\in \tH_j$ for $j\in J$, not all zero, such that
\begin{itemize}
\item[{\rm(i)}] {\rm(dual feasible)} $r_0\geq 0$ and $\lambda_i\in K_i^+$ for all $i\in I$,

\item[{\rm(ii)}] {\rm(complementary slackness)} $\langle g_i(\widehat x),\lambda_i\rangle_{\tY_i,\tY_i^*}=0$, for all $i\in I$,

\item[{\rm(iii)}] {\rm(non-trivial multipliers)} $ |r_0|+\sum_{i\in I}\lv\lambda_i\rv_{\tY_i^*}+\sum_{j\in J}\lv \mu_j\rv_{\tH_j}= 1$, where $\lv\cdot\rv_{\tH_j}$ is the norm induced by the inner product $\langle\cdot,\cdot\rangle_{\tH_j}$,

\item[{\rm(iv)}] {\rm(Fritz-John condition)} there exists $\xi\in N_C(\widehat{x})$ such that
\begin{align}\label{FJ}
	-\xi= r_0 Df(\widehat{x})+\sum_{i\in I}\lambda_i\circ Dg_i(\widehat{x})-\sum_{j\in J} \mu_j\circ Dh_j(\widehat{x}),
\end{align}
where $\lambda_i\circ Dg_i(\widehat{x})\in\tX^*$ is the composition of $\lambda_i\in\tY_i^*$ and $Dg_i(\widehat{x})\in\mathcal{L}(\tX;\tY_i)$. Similar notation holds for $\mu_j\circ Dh_j(\widehat{x})\in\tX^*$.
\end{itemize}
\end{theorem}

\begin{remark}
When $\widehat{x}\in C^{\rm o}$ in \thmref{fj_condition}, we have $N_C(\widehat{x})=\{0\}$, which yields that the element $\xi$ stated in \thmref{fj_condition}-(iv) can only be zero. {It is worth noting that the FJ condition obtained in Theorem \ref{fj_condition} relies crucially on the convexity of the set $C$. In particular, the derivation of the first-order optimality conditions of the C-MFC problem \eqref{cost_func} studied in Section \ref{sec:CSCP} is based on variational arguments that require the ability to compute the derivatives w.r.t. the control process. When the control space is convex, one can consider convex perturbations of the control and naturally define Fr\'{e}chet derivatives w.r.t. the control process. This allows us to derive the optimality conditions through the FJ condition. However, when the control space is non-convex, such perturbations are generally not admissible. In the classical SMP,  spike variation is invoked as an alternative to derive necessary optimality conditions in the non-convex case. However, it does not permit a simple notion of derivative w.r.t. the control process as in the convex case. Hence, the technique of spike variation is not suitable to derive some clean FJ conditions based on derivatives. Due to these technical obstacles, establishing the similar optimality conditions over non-convex action set remains an open problem, which will be left for future study. 

}
\end{remark}



\begin{proof}[Proof of \thmref{fj_condition}]
Let $\epsilon\in(0,1)$ and $F:\tX\to\R$ be defined by, for $x\in C$,
\begin{align*}
F(x):=\max\left\{f(x)-f(\widehat{x})+\epsilon,~ \!\!\!\!\sup_{\substack{y_i^*\in K_i^+\\ \lv y_i^*\rv_{\tY_i^*}=1}}\langle g_i(x),y_i^*\rangle_{\tY_i,\tY_i^*},~\lv h_j(x)\rv_{\tH_j};~i\in I,~j\in J\right\}.    
\end{align*}
Then, $x\mapsto F(x)$ is locally Lipschitz and bounded below by $0$. As $\widehat{x}$ solves the problem \eqref{banach_space_optimization} locally, we have $F(\widehat{x})=\epsilon$, 
and $\widehat{x}$ is an $\epsilon$-minimum of $x\mapsto F(x)$ over $C$. Thus, we can apply Ekeland's variation principle given by Theorem 1.1 in \cite{Ekeland} to conclude the existence of $z_{\epsilon}\in C$ such that the corresponding assertions in Ekeland's variation principle hold. 

Note that $F(z_{\epsilon})>0$, otherwise $z_{\epsilon}\in C$ would be a feasible point due to \lemref{dual}, which contradicts the fact that $\widehat{x}$ solves problem \eqref{banach_space_optimization} locally. By the statement under \remref{remark:Ifinite},     
one can deduce that $z_{\epsilon}$ is a local minimum for the mapping $x\mapsto F(x)+\sqrt{\epsilon}\lv x-z_{\epsilon}\rv_{\tX}+c_Fd_C(x)$ with $c_F$ being the local Lipschitz constant of the mapping $F$ near $\widehat{x}$. 

In lieu of \lemref{prop_ger_dif} and \lemref{maxdif}, there exists a convex combination $(r_0^{\epsilon},\phi_i^{\epsilon},\psi_j^{\epsilon})$ for $i\in I$ and $j\in J$ such that the set
\begin{align*}
r_0^{\epsilon}Df(z_{\epsilon})+\sum_{i\in I}\phi_i^{\epsilon} \pa \sup_{\substack{y_i^*\in K_i^+\\ \lv y_i^*\rv_{\tY_i^*}=1}}\langle g_i(z_{\epsilon}),y_i^*\rangle_{{\tY_i},{\tY_i}^*}-\sum_{j\in J}\psi_j^{\epsilon}\pa\lv h_j(z_{\epsilon})\rv_{\tH_j}+c_F\pa d_C(z_{\epsilon})   
\end{align*}
intersects $\sqrt{\epsilon} B_1^*$ where $B_1^*$ is the unit ball in $\tX^*$, and $\phi_i^{\epsilon}=0$ if $g_i(z_{\epsilon})\leq_{K_i} 0$ for $i\in I$ and $\psi_j^{\epsilon}=0$ if $h_j(z_{\epsilon})=0$ for $j\in J$. It is easy to verify that $x\mapsto\lv h_j(x)\rv_{\tH_j}$ is continuously Fr\'{e}chet differentiable and $D\lv h_j(x)\rv_{\tH_j}=(h_j(x)/\lv h_j(x)\rv_{\tH_j})\circ Dh_j(x)$ when $h_j(x)\neq0$ for $i\in I$. Thus, let us set $\mu_j^{\epsilon}:=-(h_j(z_{\epsilon})/\lv h_j(z_{\epsilon})\rv_{\tH_j})\psi_j^{\epsilon}\in \tH_j$ 
when $h_j(z_{\epsilon})=0$, and set $\mu_j^{\epsilon}:=0$ otherwise, for $j\in J$. On the other hand, we have 
$D\langle g_i(x),y_i^*\rangle_{\tY_i,\tY_i^*}=y_i^*\circ Dg_i(x)\in \tX^*$ 
for $x\in\tX$. We claim that the set $\{y_i^*\circ Dg_i(x);~y_i^*\in M(x)\}$ is closed and convex, where $M_{i}(x)$ is the maximizer set given by
\begin{align*}
M_i(x)=\left\{y_i^*\in \{y_i^*\in K_i^+;~\lv y_i^*\rv_{\tY_i^*}=1\};~\langle g_i(x),y_i^*\rangle_{\tY_i,\tY_i^*}=\sup_{\substack{y_i^*\in K_i^+\\ \lv y_i\rv_{\tY_i^*}=1}}\langle g_i(x),y_i^*\rangle_{\tY_i,\tY_i^*}\right\}. 
\end{align*}
As $\langle g_i(x),y_i^*\rangle_{\tY_i,\tY_i^*}$ is weak-$*$ continuous in $y_i^*$ for any $x\in\tX$ and $\{y_i^*\in K_i^+;~\lv y_i^*\rv_{\tY_i^*}=1\}$ is weak-$*$ compact due to Alaogu's theorem, $M_i(x)$ is nonempty, and moreover it can be easily verified to be convex and weak-$*$ closed (and hence closed). The convexity of $\{y_i^*\circ Dg_i(x);~y_i^*\in M_i(x)\}$ follows from $Dg_i(x)\in\mathcal{L}(\tX,\tY_i)$ and the convexity of $M_i(x)$. To show the closedness, let us assume that $y_i^{n^*}\circ Dg_i(x)\to x^*$ in $\tX^*$ with $y_i^{n^*}\in M_i(x)$. As the latter set is bounded, according to the separability, there exists a subsequence $(y_i^{n_k^*})$ of $(y_i^{n^*})$ weakly-$^*$ converging to some $y^*\in \tY_i^*$ (in fact $y^*\in M_i(x)$ in view that $M_i(x)$ is weak-$^*$ closed). Thus, we can conclude that
\begin{align*}
y^*\circ Dg_i(x)(x')=\lim_{n_k\to\infty}y_i^{n_k^*}\circ Dg_i(x) (x')=\langle x', x^* \rangle_{\tX,\tX^*},\quad \forall x'\in \tX. 
\end{align*}
Hence, we have $x^*= y^*\circ Dg_i(x)$. It then follows from \lemref{maxdif} again that, for $i\in I$,
\begin{align*}
\pa \sup_{\substack{y_i^*\in K_i^+\\ \lv y_i\rv_{\tY_i^*}=1}}\langle g_i(x),y_i^*\rangle_{\tY_i,\tY_i^*}&\subset\cl{\co\left\{y_i^*\circ Dg_i(x);~y_i\in M_i(x)\right\}}=\left\{y_i^*\circ Dg_i(x);~y_i\in M_i(x)\right\}.\nonumber  
\end{align*} 

We then deduce that the element in $\pa \sup_{{y_i^*\in K_i^+,~\lv y_i\rv_{\tY_i^*}=1}}\langle g_i(x),y_i^*\rangle_{\tY_i,\tY_i^*}$ can be expressed in the form of $y_i^*\circ Dg_i(x)$ with $y_i^*\in M_i(x)$, and thus $\|y_i^*\|_{\tY_i^*}=1$. Let us set $\lambda_i^{\epsilon}=\phi_i^{\epsilon}y_i^*\in K_i^+$ with $y_i^*$ given above for $i\in I$. Recall $\phi_i^{\epsilon}=0$ when $g_i(z_{\epsilon})\leq_{K_i}0$, and it thus holds that $\exists~\xi_{\epsilon}\in\tX^*$ such that
\begin{align*}
\begin{cases}
	\displaystyle -\xi_{\epsilon}= r_0^{\epsilon}Df(z_{\epsilon})+\sum_{i\in I}\lambda_i^{\epsilon}\circ Dg_i(z_{\epsilon})-\sum_{j\in J}\mu_j^{\epsilon}\circ Dh_j(z_{\epsilon}),\\[0.8em]
	\displaystyle       ~~\xi_{\epsilon}\in c_F\pa d_C(z_{\epsilon})-\sqrt{\epsilon}B_1^*,\\[0.6em]
	\displaystyle    ~~   r_0^{\epsilon}\geq 0,~\lambda_i^{\epsilon}\in K_i^+,~\lambda_i^{\epsilon}=0 \text{ if } g_i(z_{\epsilon})\leq_{K_i} 0,~ \forall i\in I,\\[0.8em]
	\displaystyle \langle g_i(z_{\epsilon}),\lambda_i^{\epsilon}\rangle_{\tY_i,\tY_i^*}=\phi_i^{\epsilon}\sup_{\substack{y_i^*\in K_i^+\\ \lv y_i^*\rv_{\tY_i^*}=1}}\langle g_i(z_{\epsilon}),y_i^*\rangle_{\tY_i,\tY_i^*}\leq F(z_{\epsilon})\leq\epsilon,~\forall i\in I,\\[1.8em]
	\displaystyle ~~r_0^{\epsilon}+\sum_{i\in I}\|\lambda_i^{\epsilon}\|_{\tY_i^*}+\sum_{j\in J}\lv\mu_j^{\epsilon}\rv_{\tH_j}=1.  
\end{cases}
\end{align*}
The last equality results from the fact that $(r_0^{\epsilon},\phi_i^{\epsilon},\psi_j^{\epsilon})$ for $i\in I$ and $j\in J$ is a convex combination and they multiply functionals with unit norms. Now, let $\epsilon\to 0$, by the definition of $z_{\epsilon}\in C$, we have $z_{\epsilon}\to\widehat{x}$. In view of the boundedness, we may assume that $(\xi_{\epsilon},\lambda_i^{\epsilon},\mu_j^{\epsilon})$ converges to $(\xi,\lambda_i,\mu_j)$ in the weak-$*$ topology and $r_0^{\epsilon}$ converges to $r_0\geq 0$ (by a subsequence if necessary) as $\epsilon\to0$. On the other hand, \lemref{prop_ger_dif}-(v) implies both \equref{FJ} and the relation $\xi\in c_F\pa d_C(\widehat{x})\subset N_C(\widehat{x})$. In addition, $\lambda_i\in K_i^+$ as $K_i^+$ is closed and convex for $i\in I$. Obviously,  \thmref{fj_condition}-(iii) holds and it suffices to show (ii). For any $i\in I$, if $g_i(z_{\epsilon})\leq_{K_i}0$, then $\langle g_i(z_{\epsilon}),\lambda_i^{\epsilon}\rangle_{\tY_i,\tY_i^*}=0$, and if $-g_i(z_{\epsilon})\notin K_i$, then \lemref{dual} results in $\langle g_i(z_{\epsilon}),\lambda_i^{\epsilon}\rangle_{\tY_i,\tY_i^*}\geq 0$. All in all, we have $|\langle g_i(z_{\epsilon}),\lambda_i^{\epsilon}\rangle_{\tY_i,\tY_i^*}|\leq \epsilon$. Moreover, we can prove $\langle g_i(\widehat{x}),\lambda_i\rangle_{Y_i,Y_i^*}=\lim_{\epsilon\to0}|\langle g_i(z_{\epsilon}),\lambda_i^{\epsilon}\rangle_{\tY_i,\tY_i^*}|=0$ as in \lemref{interior}, which yields \thmref{fj_condition}-(ii). In particular, \thmref{fj_condition}-(iii) also implies that the multipliers are not all zeros, which completes the proof.
\end{proof}

\begin{remark}
Thanks to \lemref{interior}, it follows from \thmref{fj_condition}-{\rm(ii)} that $\lambda_i=0$ when $g_i(\widehat{x})<_{K_i} 0$.
\end{remark}

{
In addition to the previous FJ necessary condition, we next also provide a sufficient condition to ensure the global optimality, whose proof is given in Appendix \ref{appendix}.
\begin{prop}[FJ sufficient condition]\label{fj_sufficient}Let $f$ be convex on $C$ and $h_j$ admit the affine form 
$h_j(x) = L_j x + e_j$ with $e_j \in \mathtt{H}_j$ and $L_j:\mathtt{X}\to\mathtt{H}_j$ being a bounded linear functional for $j\in J$. Moreover, for $i\in I$, assume that $g_i$ is convex in $x\in C$ in the sense that
\begin{align*}
g_i(\lambda x_1 + (1-\lambda) x_2) \le_{K_i} \lambda g_i(x_1) + (1-\lambda) g_i(x_2),~~ 
\forall \lambda \in [0,1],\ x_1,x_2\in C.
\end{align*}
Then, any feasible point $\widehat{x}\in C$ satisfying the necessary condition in \thmref{fj_condition} with $r_0>0$ is a global minimizer of the constrained optimization problem \eqref{banach_space_optimization}.
\end{prop}}

\subsection{Constraint Qualifications}

The goal of this subsection is to discuss the constraint qualification (CQ) conditions under which our generalized Fritz-John optimality condition can be enhanced to be the Karush-Kuhn-Tucker (KKT) type condition, which plays an important role in studying the classical constrained stochastic control problems in Section~\ref{sec:constrainedSCP-MFG}.

Let us assume that the space $\tX$ admits the decomposition $\tX=\tX_1\times\tX_2$, where both $(\tX_1,\|\cdot\|_{\tX_1})$ and $(\tX_2,\|\cdot\|_{\tX_2})$ are Banach spaces. The norm on $\tX$ is defined to be the product norm $\|(x_1,x_2)\|_{\tX}=\|x_1\|_{\tX_1}+\|x_2\|_{\tX_2}$ for $(x_1,x_2)\in\tX_1\times\tX_2$. We also define the natural projection mapping $\Pi_1:\tX\to\tX_1$ which maps $x=(x_1,x_2)\in\tX=\tX_1\times\tX_2$ to $x_1\in\tX_1$.
Then, we introduce the following definitions:
\begin{definition}\label{CQ}
Let $\widehat{x}=(\widehat{x}_1,\widehat{x}_2)\in C$ be feasible for problem \eqref{banach_space_optimization}, and set $I(\widehat{x}):=\{i\in I;~g_i(\widehat{x})\in\pa K_i\}$. We say that
\begin{itemize}
\item [{\rm(i)}] the linear independence constraint qualification {\rm(LICQ)} holds at $\widehat{x}$ {\rm(}write ${\rm LICQ}(\widehat{x})${\rm)} if the following mapping is surjective:
\begin{align*}
	(D_{x_1}g_i(\widehat{x}_1),D_{x_1}h_j(\widehat{x}_1))_{i\in I(\widehat{x}),j\in J}:\tX_1\mapsto\prod_{i\in I(\widehat{x})}\tY_i\times\prod_{j\in J}\tH_j  ;
\end{align*}

\item [{\rm(ii)}] the Mangasarian-Fromovitz constraint qualification {\rm(MFCQ)} holds at $\widehat{x}$ {\rm(}write ${\rm MFCQ}(\widehat{x})${\rm)} if the mapping $ (D_{x_1}h_j(\widehat{x}_1))_{j\in J}:\tX_1\mapsto\prod_{j\in J}\tH_j$ is surjective, and there exists an element $c\in C$ such that
\begin{align*}
	Dg_i(\widehat{x})(c-\widehat{x})<_{K_i}0,~\forall i\in I(\widehat{x}),\quad Dh_j(\widehat{x})(c-\widehat{x})=0,~\forall j\in J.
\end{align*}
\end{itemize}
\end{definition}

\begin{remark}
If $\tX$ is a Euclidean space, $\tX_1=\tX$ and $\tX_2=\{0\}$, the CQ conditions defined in \defref{CQ} reduce to {\rm LICQ} and {\rm MFCQ} in the Euclidean space, respectively.
\end{remark}

Next, we also provide a sufficient condition for the KKT optimality condition to hold. Let $\Pi_1(C):=\{\Pi_1(x);~x\in C\}$. Then, we have the following corollary whose proof is reported in Appendix~\ref{appendix}:  
\begin{corollary}[KKT condition]\label{CQcon}
Let $\widehat{x}\in C$ solve problem \eqref{banach_space_optimization} locally and $\Pi_1(\widehat{x})\in(\Pi_1(C))^{\rm o}$ in the relative topology of $\tX_1$. If either ${\rm LICQ}(\widehat{x})$ or ${\rm MFCQ}(\widehat{x})$ is fulfilled, the KKT optimality condition holds true in the sense that there exist $\lambda_i\in \tY_i^*$ and $\mu_j\in \tH_j$ for $i\in I$ and $j\in J$ such that
\begin{itemize}
\item[{\rm(i)}] {\rm(dual feasible)} $\lambda_i\in K_i^+$ for all $i\in I$,

\item[{\rm(ii)}] {\rm(complementary slackness)} $\langle g_i(\hat x),\lambda_i\rangle_{\tY_i,\tY_i^*}=0$ for all $i\in I$,

\item[{\rm(iii)}] {\rm(KKT condition)} there exists $\xi\in N_C(\widehat{x})$ such that
\begin{align*}
	-\xi=  Df(\widehat{x})+\sum_{i\in I}\lambda_i\circ Dg_i(\widehat{x})-\sum_{j\in J} \mu_j\circ Dh_j(\widehat{x}).
\end{align*}
\end{itemize}
\end{corollary}
{
We have the following remarks on some distinctions and connections between the FJ condition and KKT condition.
\begin{remark}\label{FJvsKKT}
The Fritz-John (FJ) condition and Karush-Kuhn-Tucker (KKT) condition are closely related but rely on different assumptions. 
The FJ condition holds without any constraint qualification, which makes them applicable in general settings. However, they may lead to degenerate multipliers with $r_0=0$ in \equref{FJ}. By contrast, the KKT condition requires a constraint qualification such as the linear independence constraint qualification (LICQ) or the Mangasarian-Fromovitz constraint qualification (MFCQ). 
Under these assumptions, the multiplier associated with the objective function can be normalized (typically $r_0=1$), yielding a more interpretable optimality system.

For our C-MFC problem \equref{cost_func} (or equivalently \equref{optimization}) to be studied in the next section, it is generally difficult to verify such constraint qualifications, because the constraints are imposed on the state dynamics and the control process. As a result, we mainly apply the FJ condition in \thmref{fj_condition} to analyze the C-MFC problem. 
Nevertheless, as stated in \remref{rem:classicalSMP}, when the inequality constraints in \equref{cost_func} are absent, the LICQ/MFCQ conditions are automatically satisfied, and the FJ condition is enhanced to be the KKT condition for the unconstrained MFC problem.
\end{remark}
}

\section{Extended MFC Problems with Constraints} \label{sec:CSCP}

In this section, we apply the newly obtained FJ optimality condition in \thmref{fj_condition} to establish the SMP for some C-MFC problems in which we can encompass both dynamic expectations constraints and dynamic pathwise state-control and law constraints. To this end, we first formulate our C-MFC problem and introduce some basic spaces and notations. 

Let $(\Omega,\F,\Fb,\Pb)$ be a filtered  probability space with the filtration $\Fb=(\F_t)_{t\in[0,T]}$ satisfying the usual conditions. For $r,l\in\mathbb{N}$, let $(W_t)_{t\in[0,T]}$ be a standard $r$-dimensional $(\Pb,\Fb)$-Brownian motion and denote $U\subset\R^l$ as the closed and convex control space. We  define $\X$ as the space $L^2([0,T];\R^l)/\mathord{\sim}$ equipped with the inner product $\langle\gamma^1,\gamma^2\rangle_{\X}=\int_0^T\langle\gamma_t^1,\gamma_t^2\rangle\d t$. Then,  $(\X,\langle\cdot,\cdot\rangle_{\X})$ is a Hilbert space. For any $\gamma=(\gamma_t)_{t\in[0,T]}\in\X$, we introduce, for fixed $t\in[0,T]$, 
\begin{align}\label{tilde}
\widetilde{\gamma}^t_{s}:=\gamma_s\mathbf{1}_{s\in[0,t]},\quad \forall s\in[0,T].
\end{align}
Then $\|\widetilde{\gamma}^t\|_{\X}=\sqrt{\langle\widetilde{\gamma}^t,\widetilde{\gamma}^t\rangle_{\X}}\leq\|\gamma\|_{\X}$ for all $t\in[0,T]$. Denote by $\U[0,T]$ the set of admissible controls $\alpha=(\alpha_t)_{t\in[0,T]}$ valued in $U$ such that the process $\alpha$ is $\Fb$-adapted and satisfies $\E[\int_0^T|\alpha_t|^2\d t]<\infty$. Let $\Lb_{\Fb}^2$ be the set of $\Fb$-adapted processes $\beta=(\beta_t)_{t\in[0,T]}$ taking values in $\R^l$ such that $\E[\int_0^T|\beta_t|^2\d t]<\infty$.  Then, $\mathbb{L}_{\Fb}^2$ is a Hilbert space equipped with the inner product that
\begin{align}\label{eq:innerproLF2}
\langle \beta^1,\beta^2\rangle_{\Lb_{\Fb}^2}=\E\left[\int_0^T\langle\beta^1_t,\beta^2_t\rangle\d t\right],\quad \forall  \beta^1,\beta^2\in\mathbb{L}_{\Fb}^2, 
\end{align}
where $\beta^1=\beta^2$ means that $\beta^1(\omega)=\beta^2(\omega)$ holds in $\X$ except for a $\Pb$-null set. It can be easily verified that $\U[0,T]$ is closed and convex in $(\Lb_{\Fb}^2,\langle\cdot,\cdot\rangle_{\Lb_{\Fb}^2})$.
We can observe that $\beta\in L^2(\Omega;\X)$ for all $\beta\in\Lb_{\Fb}^2$, and moreover, we have $\|\beta\|_{\Lb_{\Fb}^2}^2:=\E\left[\|\beta\|_{\X}^2\right]=\E\left[\langle\beta,\beta\rangle_{\X}\right]$.
Let us define $\mathbb{M}^2$ ($\subset\Lb_{\Fb}^2$) as the set of $\Fb$-adapted (continuous) $\R^n$-valued process $\eta=(\eta_t)_{t\in[0,T]}$ satisfying the It\^o's representation that 
\begin{align}\label{eq:ito-rep}
\eta_t=\eta_0+\int_0^t\eta^1_s\d s+\int_0^t\eta_s^2\d W_s, 
\end{align}
where $\eta_0\in L^2(\Omega;\R^n)$ is $\F_0$-measurable,  $\eta^1=(\eta_t^1)_{t\in[0,T]}$ and $\eta^2=(\eta_t^2)_{t\in[0,T]}$ take values respectively in $\R^n$ and $\R^{n\times r}$, which are $\Fb$-adapted processes such that $\E[\int_0^T(|\eta_t^1|^2+|\eta_t^2|^2)\d t]<\infty$. Then, $\mathbb{M}^2$ becomes a Hilbert space under the inner product $\langle\cdot,\cdot\rangle_{\mathbb{M}^2}$ defined by
\begin{align}\label{eq:innerM2}
\langle \eta,\zeta\rangle_{\mathbb{M}^2}:=\E\left[\langle\eta_0,\zeta_0\rangle+\int_0^T\left(\langle\eta_t^1,\zeta_t^1\rangle+\tr\left(\eta_t^2(\zeta_t^2)^{\T}\right)\right)\d t \right],\quad\forall\eta,\zeta\in\mathbb{M}^2.    
\end{align}	
We use $\Lb_{\Fb,\R}^2$ to represent the set of adapted processes in $L^2([0,T]\times\Omega;\R)/\mathord{\sim}$. Then, $\Lb^2_{\Fb,\R}$ becomes a Hilbert space when it is equipped with the inner product $\langle\cdot,\cdot\rangle_{\Lb^2_{\Fb,\R}}$ defined by
\begin{align}\label{eq:L2}
\langle\beta^1,\beta^2\rangle_{\Lb^2_{\Fb,\R}}:=\E\left[\int_0^T\beta_t^1\beta_t^2\d t\right],\quad \forall \beta^1,\beta^2\in\Lb^2_{\Fb,\R}.
\end{align}
We also denote $\|\cdot\|_{\Lb_{\Fb,\R}^2}$ as the  induced norm by above inner product.

\subsection{Partial L-derivative}

As a technical preparation, we follow Acciaio et al. \cite{Carmona1} to introduce the partial L-derivative of joint probability laws (i.e. probability measures on product spaces). 

Let $w:\Pc_2(\R^n\times\R^l)\mapsto\R$. We use $\rho$ to represent a generic element in $\Pc_2(\R^n\times\R^l)$, while $\mu\in\Pc_2(\R^n)$ and $\nu\in\Pc_2(\R^l)$ denote its respective marginals. For a probability space $(\tilde{\Omega},\tilde{\F},\tilde{\Pb})$, define the lift $\tilde w:L^2((\tilde{\Omega},\tilde{\F},\tilde{\Pb});\R^n\times\R^l)\to \R$ of the element $w$ by $\tilde w(\tilde X,\tilde\alpha):=w(\Law(\tilde X,\tilde\alpha))$, 
where $\Law(\tilde{X},\tilde{\alpha})$ denotes the joint distribution of $(\tilde{X},\tilde{\alpha})$ under $\tilde P$. Then, the L-derivative is defined by
\begin{definition}\label{def:L-derivative}
We call $w$ the {\rm L}-differentiable if there is a pair $(\tilde X,\tilde\alpha)\in L^2((\tilde{\Omega},\tilde{\F},\tilde{\Pb});\R^n\times\R^l)$ with the joint distribution $\Law^{\tilde{P}}(\tilde X,\tilde\alpha)$ such that the lifted function $\tilde w$ is Fr\'{e}chet differentiable at $(\tilde X,\tilde\alpha)$.
\end{definition}

By Section 5.2.1 in Carmona and Delarue \cite{Carmona2}, there exists a measurable mapping $\pa_{\rho}w(\rho)(\cdot):\R^n\times\R^l\mapsto\R^n\times\R^l$ such that $\pa_{\rho}w(\rho)(\tilde X,\tilde\alpha)=D\tilde w(\tilde X,\tilde\alpha)$. Since $L^2((\tilde{\Omega},\tilde{\F},\tilde{\Pb});\R^n\times\R^l)\simeq L^2((\tilde{\Omega},\tilde{\F},\tilde{\Pb});\R^n)\times L^2((\tilde{\Omega},\tilde{\F},\tilde{\Pb});\R^l)$, we can define the first component and second component of $\pa_{\rho}w(\rho)(\tilde X,\tilde\alpha)$ as $\pa_{\mu}w(\rho)(\tilde X,\tilde\alpha)$ and $\pa_{\nu}w(\rho)(\tilde X,\tilde\alpha)$, respectively. By the measurability of $\pa_{\rho}w(\rho)(\tilde X,\tilde\alpha)$, we can also conclude that $\pa_{\mu}w(\rho)(\cdot):\R^n\times\R^l\mapsto \R^n$ and $\pa_{\nu}w(\rho)(\cdot):\R^n\times\R^l\mapsto \R^l$ are both measurable. We call $\pa_{\mu}w(\rho)$ and $\pa_{\nu}w(\rho)$ the partial derivatives of $w$. By Definition~\ref{def:L-derivative} on the L-derivative, we have the following lemma.
\begin{lemma}\label{Lpartialdif}
Let $\rho=\Law(\tilde X,\tilde \alpha)$ for a  $\R^n\times\R^l$-valued r.v. $(\tilde X,\tilde\alpha)$ on some probability space $(\tilde\Omega,\tilde\F,\tilde P)$. Then, for any $\R^n\times\R^l$-valued r.v. $(\tilde X',\tilde \alpha')$, it holds that
\begin{align*}
w(\Law(\tilde X',\tilde\alpha'))&=w(\rho)+\E^{\tilde P}\left[\langle \pa_{\mu}w(\rho)(\tilde X,\tilde\alpha),\tilde X'-\tilde X\rangle\right]+\E^{\tilde P}\left[\langle \pa_{\nu}w(\rho)(\tilde X,\tilde\alpha),\tilde \alpha'-\tilde \alpha\rangle\right]\\
&\quad+o\left(\sqrt{\E^{\tilde P}[|(\tilde X',\tilde\alpha')-(\tilde X,\tilde\alpha)|^2]}\right), 
\end{align*}
when $(\tilde X',\tilde \alpha')$ tends to $(\tilde X,\tilde\alpha)$ in $L^2((\tilde\Omega,\tilde\F,\tilde P);\R^n\times\R^l)$. Here, $\E^{\tilde{P}}$ denotes the expectation operator under $\tilde{P}$, and the above $\langle\cdot,\cdot\rangle$ is the inner product of $L^2((\tilde\Omega,\tilde\F,\tilde P);\R^n\times\R^l)$.
\end{lemma}

In the sequel, we will construct a copy space $(\Omega',\F',\Pb')$ of the probability space $(\Omega,\F,\Pb)$ in the sense that $(\Omega',\F',\Pb')=(\Omega,\F,\Pb)$. For a random variable $X$ constructed on $(\Omega,\F,\Pb)$, we denote by $X'$ its copy random variable on $(\Omega',\F',\Pb')$ such that $X'(\omega')=X(\omega)$ for $\omega'=\omega\in\Omega$.

\subsection{Formulation of C-MFC problems}

In this section, we formulate our C-MFC problems. Let $(b,\sigma):[0,T]\times\R^n\times{\cal P}_{2}(\R^n\times\R^l)\times\R^l\mapsto\R^n\times\R^{n\times r}$ be measurable mappings. For any $\alpha=(\alpha_t)_{t\in[0,T]}\in\U[0,T]$, consider the following controlled Mckean-Vlasov SDE that, $X_0^{\alpha}=\kappa\in L^2((\Omega,\F_0,\Pb);\R^n)$, and for $t\in(0,T]$,
\begin{align}\label{eq:controlledSDE}
\d X_t^{\alpha} &=b(t,X_t^{\alpha},\Law(X_t^{\alpha},\alpha_t),\alpha_t)\d t +\sigma(t,X_t^{\alpha},\Law(X_t^{\alpha},\alpha_t),\alpha_t)\d W_t.
\end{align}	
Let integer $1\leq m_1\leq m$. For $\alpha\in\U[0,T]$ and $\phi^i:[0,T]\times\R^n\times\Pc_2(\R^n)\times\X\mapsto\R$ with $i\in I_1:=\{1,2,\ldots,m_1\}$, we introduce the dynamic expectation constraints in the sense that, for $i\in I_1$, 
\begin{align}\label{expected_constraint}
\E\left[\phi^i\left(t,X_t^{\alpha},\Law(X_t^{\alpha}),\widetilde{\alpha}^t\right)\right]\geq 0,\quad \forall t\in[0,T].
\end{align}

For $\psi^j:[0,T]\times \R^n\times\Pc_2(\R^n\times\R^l)\times\R^l\mapsto\R$ with $j\in I_2:=I\setminus I_1$, the dynamic pathwise state-control and joint law constraints are formulated by, for $j\in I_2$,
\begin{align}\label{path_constraint}
\psi^j\left(t,X_t^{\alpha},\Law(X_t^{\alpha},\alpha_t),\alpha_t\right)\geq 0,\quad \d t\times\d\Pb\text{-}\as.
\end{align}
Here, $X^{\alpha}=(X_t^{\alpha})_{t\in[0,T]}$ obeys the controlled Mckean-Vlasov SDE \eqref{eq:controlledSDE} and the process $\widetilde{\alpha}^t$ is defined by \eqref{tilde}. Then, we introduce the admissible control set $\U_{ad}[0,T]$ as the set of processes $\alpha\in\U[0,T]$ such that there exists $X^{\alpha}=(X_t^{\alpha})_{t\in[0,T]}\in\mathbb{M}^2$ satisfying \equref{eq:controlledSDE}, \equref{expected_constraint} and \equref{path_constraint}.
Based on the preparations provided above, we introduce the cost functional as follows, for $\alpha=(\alpha_t)_{t\in [0,T]}\in\U_{ad}[0,T]\neq\varnothing$, 
\begin{align}\label{cost_func0}
J(\alpha):=\E\left[\int_0^Tf(t,X_t^{\alpha},\Law(X_t^{\alpha},\alpha_t),\alpha_t)\d t+g(X_T^{\alpha},\Law(X_T^{\alpha}))\right].
\end{align}
\begin{remark}
The expectation constraint \eqref{expected_constraint} involving the law $\Law(X_t^{\alpha})$ is motivated by the variance-type constraint $\E[|X_t^{\alpha}+\int_0^t\alpha_s\d s-\E [X_t^{\alpha}]|^2]\leq a_t$, where $t\mapsto a_t$ is a deterministic continuous mapping. This type of constraints cannot be encompassed by the pathwise constraint in \eqref{path_constraint}. Moreover, when $\phi^i$
is independent of the measure argument $\mu$, the constraint \eqref{expected_constraint} extends the expectation constraint considered by Hu et al. \cite{Hu} in which the control does not enter into the constraint.
\end{remark}

The following assumptions on model coefficients are imposed throughout the section:

\begin{ass}\label{ass1} 
\begin{itemize}
\item[{\rm(A1)}] for $i\in I_1$ and $j\in I_2$, the mappings $b,\sigma,f,g,\phi^i$ and $\psi^j$ are Borel measurable and continuously differentiable w.r.t. $x\in\R^n$. Moreover, $b,\sigma,f,\psi^j$ are continuously differentiable w.r.t. $u\in\R^l$ and $\phi^i$ is continuously Fr\^echet differentiable w.r.t. $\gamma\in\X$. Moreover, $\phi^i$ is jointly continuous in $(t,x,\mu)\in[0,T]\times\R^n\times\Pc_2(\R^n)$.

\item[{\rm(A2)}] there is a constant $M>0$ such that, for all $(t,x,\mu,\rho,u,\gamma)\in[0,T]\times\R^n\times\Pc_2(\R^n)\times\Pc_2(\R^n\times\R^l)\times\R^l\times\X$,
\begin{align*}
	\sum_{i\in I_1}\left(|\nabla_x\phi^i(t,x,\mu,\gamma)|+\|D_{\gamma}\phi^i(t,x,\mu,\gamma)\|_{\X}\right)&\leq M\left(1+|x|+M_2(\mu)+\|\gamma\|_{\X}\right),\\
	\sum_{j\in I_2}\left(|\nabla_x\psi^j(t,x,\rho,u)|+|\nabla_u\psi^j(t,x,\rho,u)|\right)&\leq M.
\end{align*}
Here $M_2(\mu):=\int_{\R^n}|x|^2\mu(\d x)<\infty$.
\item[{\rm(A3)}] the mappings $b(t,x,\rho,u)$ and $\sigma(t,x,\rho,u)$ are uniformly Lipschitz continuous in $(x,\rho)\in\R^n\times\Pc_2(\R^n\times \R^l)$ in the sense that, there is a constant $M>0$ independent of $u\in \R^l$ such that, for all $(x,\rho),(x',\rho')\in\R^n\times\Pc_2(\R^n\times \R^l)$ and $(t,u)\in[0,T]\times \R^l$,
\begin{align*}
	\left|b(t,x',\rho',u)-b(t,x,\rho,u)\right|+\left|\sigma(t,x',\rho',u)-\sigma(t,x,\rho,u)\right|\leq M\left(|x-x'|+\Wc_2(\rho,\rho')\right),
\end{align*}
where $\Wc_2(\cdot,\cdot)$ denotes the Wasserstein metric on ${\cal P}_2(\R^n\times \R^l)$.

\item[\rm(A4)] there is a constant $M>0$ such that, for all $(t,x,\rho,u)\in[0,T]\times\R^n\times\Pc_2(\R^n\times \R^l)\times \R^l$, it holds that 
\begin{align*}
	|\delta(t,x,\rho,u)|\leq M(1+|x|+|u|+M_2(\rho)),
\end{align*}
where $\delta\in\{b,\sigma,\nabla_ub(\cdot),\nabla_u\sigma(\cdot),\nabla_xf(\cdot),\nabla_uf(\cdot)\}$ and $M_2(\rho):=\int_{\R^n\times\R^l}(|x|^2+|u|^2)\rho(\d x,\d u)<\infty$. Moreover, for all $(x,\mu)\in\R^n\times\Pc_2(\R^n)$,
\begin{align*}
	|\nabla_xg(x,\mu)|\leq M(1+|x|+M_2(\mu)).
\end{align*}
\item[\rm(A5)] for $i\in I$ and $j\in J$, the mappings $b(t,x,\rho,u),\sigma(t,x,\rho,u),f(t,x,\rho,u)$, $g(x,\mu)$, $\phi^i$ and $\psi^j$ are continuously {\rm L}-differentiable w.r.t. $\rho\in\Pc_2(\R^n\times \R^l)$ or $\mu\in\Pc_2(\R^n)$. Moreover, there is a constant $M>0$ such that, for all $(t,x,\rho,\mu,u,\gamma)\in[0,T]\times\R^n\times\Pc_2(\R^n\times \R^l)\times\Pc_2(\R^n)\times \R^l\times\X$,
\begin{align*}
	& \left(\int_{\R^n\times\R^l}|\pa_{\rho}f(t,x,\rho,u)(x',u')|^2\rho(\d x',\d u')\right)^{\frac12}+\left(\int_{\R^n}|\pa_{\mu}g(x,\mu)(x')|^2\mu(\d x')\right)^{\frac12}\\
	&~+\sum_{i\in I_1}\left(\int_{\R^n}|\pa_{\mu}\phi^i(t,x,\mu,\gamma)(x')|^2\mu(\d x) \right)^{\frac12} \leq M(1+|x|+|u|+M_2(\rho)+M_2(\mu)+\|\gamma\|_{\X}),
\end{align*} 
and moreover
\begin{align*}
	\sum_{j\in I_2}\left(\int_{\R^n\times\R^l}|\pa_{\rho}\psi^j(t,x,\rho,u)(x',u')|^2\rho(\d x',\d u')\right)^{\frac12}\leq M.
\end{align*}
\end{itemize}
\end{ass}

We first have the following remark:
\begin{remark}\label{rem:boundedLdif}
Under \assref{ass1}, for any $\alpha=(\alpha_t)_{t\in[0,T]}\in\U_{ad}[0,T]$, the Mckean-Vlasov SDE \eqref{eq:controlledSDE} will always admit a unique solution $X^{\alpha}=(X_t^{\alpha})_{t\in[0,T]}$ satisfying $\E[\sup_{t\in[0,T]}|X_t^{\alpha}|^2]<\infty$. Thus,  $X\in\M^2$ and the cost functional $J(\alpha)$ defined in \eqref{cost_func} are well-posed. Moreover, it follows from  \assref{ass1}-{\rm(A2)} that  $\E[\phi^i(t,X_t,\Law(X_t),\widetilde{\alpha}^t)]<\infty$ for all $t\in[0,T]$ since $\phi^i$ has at most quadratic growth in $(x,\gamma)\in\R^n\times\X$. On the other hand, by using \assref{ass1}-${\rm (A3)}$, since $b(t,x,\rho,u)$ and $\sigma(t,x,\rho,u)$ are continuously {\rm L}-differentiable w.r.t. $\rho\in{\cal P}_2(\R^n\times\R^l)$, it holds that, for all $(t,x,\rho,u)\in[0,T]\times\R^n\times\Pc_2(\R^n\times\R^l)\times\R^l$,
{\small\begin{align*}
	& \left(\int_{\R^n\times\R^l}|\pa_{\rho}b(t,x,\rho,u)(x',u')|^2\rho(\d x',\d u')\right)^{\frac12}+\left(\int_{\R^n\times\R^l}|\pa_{\rho}\sigma(t,x,\rho,u)(x',u')|^2\rho(\d x',\d u')\right)^{\frac12}\leq M,   
	\end{align*}}
	for the positive constant $M$ given in \assref{ass1}-{\rm(A3)}. 
\end{remark}

\begin{remark}\label{compare}
Comparing with Acciaio et al.~\cite{Carmona1} or  Hu et al. \cite{Hu}, we do not impose the uniform boundedness on gradients $\nabla_u(b,\sigma)$ of coefficients $(b,\sigma)$ of the state dynamics, which is relaxed to be  the linear growth condition in \assref{ass1}-{\rm(A4)}. 
\end{remark}

We then consider the following C-MFC problem given by
\begin{align}\label{cost_func}
\begin{cases}
\text{minimize $J(\alpha)$ over $\alpha\in\U[0,T]$;}\\[0.6em]
\displaystyle \text{subject to $X_t^{\alpha}=\kappa+\int_0^tb(s,X_s^{\alpha},\Law(X_s^{\alpha},\alpha_s),\alpha_s)\d s$}\\[0.8em]
\displaystyle ~~~~~~~~~~~~~~~~~~~~~~~\text{$+\int_0^t\sigma(s,X_s^{\alpha},\Law(X_s^{\alpha},\alpha_s),\alpha_s)\d W_s$,~~$\forall t\in[0,T]$,}\\[0.8em]
\displaystyle ~~~~~~~~~~~~~~~\text{$\E\left[\phi^i\left(t,X_t^{\alpha},\Law(X_t^{\alpha}),\widetilde{\alpha}^t\right)\right]\geq 0$,~~$i\in I_1$,~~$\forall t\in[0,T]$,}\\[0.8em]
\displaystyle ~~~~~~~~~~~~~~~\text{$\psi^j\left(t,X_t^{\alpha},\Law(X_t^{\alpha},\alpha_t),\alpha_t\right)\geq 0$,~~$j\in I_2$,~~$\d t\times\d\Pb$-a.s.}
\end{cases}
\end{align}
Obviously, it is equivalent to minimizing the objective functional $J(\alpha)$ in \equref{cost_func} over $\U_{ad}[0,T]$. We first have the following preliminary results	on constraints whose proof is reported in Appendix~\ref{appendix}.
\begin{lemma}\label{continuous_ell}
Let $i\in I_1$ and $(X,\alpha)\in\M^2\times\Lb_{\Fb}^2$.  Define $g_i(t):=\E[\phi^i(t,X_t,\Law(X_t),\widetilde{\alpha}^t)]$ with $t\in[0,T]$. Then it holds that $g_i\in C([0,T];\R)$.
\end{lemma}
Now, let us introduce that
\begin{align}\label{eq:XYHexam}
\tX=\M^2\times\Lb_{\Fb}^2,~ \tY_i=C([0,T];\R),~\forall i\in I_1,~\tY_{j}=\Lb^2_{\Fb,\R},~\forall j\in I_2,~\tH=\M^2,
\end{align}
where, for $x_1=(X_1,\alpha_1),x_2=(X_2,\alpha_2)\in \tX$, we define the inner product by $\langle x_1,x_2\rangle_{\tX}:=\langle X_1,X_2\rangle_{\M^2}+\langle\alpha_1,\alpha_2\rangle_{\Lb_{\Fb}^2}$.
Let $\lv\cdot\rv_{\tX}:= \sqrt{\langle \cdot,\cdot\rangle_{\tX}}$ be the induced norm. We also take $\|\cdot\|_{\tY_i}$~$(i\in I_1)$, $\|\cdot\|_{\tY_{j}}$ $(j\in I_2)$ and $\langle\cdot,\cdot\rangle_{\tH}$ as the supremum norm $\|\cdot\|_{\infty}$ on $C([0,T];\R^n)$, the norm $\|\cdot\|_{\Lb^2_{\Fb,\R}}$ (c.f. \eqref{eq:L2})  and the inner product $\langle\cdot,\cdot\rangle_{\M^2}$ (c.f. \equref{eq:innerM2}), respectively. 
Furthermore, we define $J:\tX\to\R$, $\Phi_i:\tX\to\tY_i$ ~$(i\in I_1)$, $\Psi_j:\tX\to\tY_{j}$~$(j\in I_2)$ and $F:\tX\to\tH$ as follows, for $(X,\alpha)\in\tX$,
\begin{align}\label{func_def}
\begin{cases}
\displaystyle J(X,\alpha):=\E\left[\int_0^Tf(t,X_t,\Law(X_t,\alpha_t),\alpha_t)dt+g(X_T,\Law(X_T))\right],\\[1em] 
\displaystyle \Phi_i(X,\alpha)(t):=-\E[\phi^i(t,X_{t},\Law(X_t^{\alpha}),\widetilde{\alpha}^t)],\quad \forall t\in[0,T],~i\in I_1,\\[1em] 
\displaystyle \Psi_j(X,\alpha)(t):=-\psi^j(t,X_t,\Law(X_t^{\alpha},\alpha_t),\alpha_{t}),\quad \forall t\in[0,T],~j\in I_2,\\[1em]
\displaystyle F(X,\alpha)(\cdot):=X_{\cdot}-\kappa-\int_0^{\cdot}b(s,X_s,\Law(X_s,\alpha_s),\alpha_s)\d s\!-\!\int_0^{\cdot}\sigma(s,X_s,\Law(X_s,\alpha_s),\alpha_s)\d W_s.
\end{cases}
\end{align}
The mappings $\Phi_i$ {\rm ($i\in I_1$)} and $\Psi_j~(j\in I_2)$ in \eqref{func_def} are well-defined thanks to \assref{ass1}-{\rm (A2)} and \lemref{continuous_ell}. Thus, we can formulate the optimization problem with constraints that
\begin{align}\label{optimization}
\begin{cases}
\displaystyle       \inf~J(X,\alpha)\\[0.4em]
\displaystyle       \st~~(X,\alpha)\in C:=\M^2\times\U[0,T],\\[0.4em]
\displaystyle      ~~~~~~  \Phi_i(X,\alpha)\leq_{K_1}0,~~\forall i\in I_1,\\[0.4em]
\displaystyle  ~~~~~~\Psi_j(X,\alpha)\leq_{K_2}0,~~\forall j\in I_2,\\[0.4em]
\displaystyle  ~~~~~~F(X,\alpha)=0,
\end{cases}
\end{align}
Here, $C=\M^2\times\U[0,T]$ is a closed and convex set in $\tX$, and 
\begin{align}\label{eq:K12}
\begin{cases}
\displaystyle  K_1:=\{f\in C([0,T];\R);~f(t)\geq 0,~\forall t\in[0,T]\},\\[0.6em]
\displaystyle K_2:=\{f\in\Lb^2_{\Fb,\R};~f(t,\omega)\geq 0,~\forall (t,\omega)\in[0,T]\times\Omega\}=:\Lb_{\Fb,\R+}^2    
\end{cases}
\end{align}
are closed and convex {\it cones} in $C([0,T];\R)$ and $\Lb^2_{\Fb,\R}$, respectively. Moreover, one can verify that $K_1^{\rm o}=\{f\in C([0,T];\R);~f(t)> 0,~\forall t\in[0,T]\}$ and $K_2^{\rm o}=\varnothing$. On the other hand, the well-posedness of the controlled Mckean-Vlasov SDE \equref{eq:controlledSDE} implies the following crucial equivalence whose proof is straightforward, and hence omitted.
\begin{lemma}\label{equivalence}
The value function of the primal control problem \eqref{cost_func} and infimum of the new optimization problem \eqref{optimization} coincide. Moreover, if the minimizer exists for one problem, so does the other one and two minimizers also coincide.
\end{lemma}
Consequently, it is equivalent for us to consider the abstract optimization problem \equref{optimization}. 

\subsection{FJ optimality condition for abstract problem \eqref{optimization}}

We next apply \thmref{fj_condition} to our C-MFC problem in \eqref{optimization} and derive the SMP for the necessary optimality condition. To this end, we need to provide the representations of Fr\'{e}chet derivatives of the mappings $J,~\Phi_i~(i\in I_1),~\Psi_j~(j\in I_2)$ and $F$ defined in \eqref{func_def}. In the sequel, let us set $\mu_t=\Law(X_t)$ and $\rho_t=\Law(X_t,\alpha_t)$ for $t\in[0,T]$.

\begin{lemma}\label{Frechet_dif}
The mappings $J$, $\Phi_i$ {\rm($i\in I_1$)}, $\Psi_j~(j\in I_2)$ and $F$ defined in \eqref{func_def} are continuously Fr\^{e}chet differentiable and their Fr\^{e}chet derivatives have the following representations that, for $(t,(X,\alpha),L,K)\in[0,T]\times\tX\times\M^2\times\Lb_{\Fb}^2$, $i\in I_1$, and $j\in I_2$,
\begin{align*}
D_X&J(X,\alpha)(L)=\E\left[\int_0^T\langle \nabla_xf(t,X_t,\rho_t,\alpha_t),L_t\rangle\d t+\langle \nabla_x g(X_T,\mu_T),L_T\rangle\right.\nonumber\\
&\qquad+\left.\E'\left[\int_0^T\langle\pa_{\mu}f(t,X_t,\rho_t,\alpha_t)(X_t',\alpha_t'),L_t'\rangle\d t+\langle\pa_{\mu}g(X_T,\mu_T)(X_T'),L_T'\rangle\right]\right],\\
D_X&\Phi_i(X,\alpha)(L)(t)=-\E\left[\langle\nabla_x\phi^i(t,X_{t},\mu_t,\widetilde{\alpha}^t),L_t\rangle+\E'\left[\langle\pa_{\mu}\phi^i(t,X_t,\mu_t,\widetilde{\alpha}^t)(X_t'),L_t'\rangle \right] \right],\\
D_X&\Psi_i(X,\alpha)(L)(t)=-\langle \nabla_x\psi^j(t,X_t,\rho_t,\alpha_t),L_t\rangle-\E'\left[\langle\pa_{\mu}\psi^j(t,X_t,\rho_t,\alpha_t)(X_t',\alpha_t'),L_t'\rangle \right],\\[0.4em]
D_X&F(X,\alpha)(L)(\cdot)=L_{\cdot}-\int_0^{\cdot}\nabla_xb(s,X_s,\rho_s,\alpha_s)L_s\d s-\int_0^{\cdot}\nabla_x\sigma(s,X_s,\rho_s,\alpha_s)L_s\d W_s\nonumber\\
&\quad-\int_0^{\cdot}\E'\left[\pa_{\mu}b(s,X_s,\rho_s,\alpha_s)(X_s',\alpha_s')L_s'\right]\d s-\int_0^{\cdot}\E'\left[\pa_{\mu}\sigma(s,X_s,\rho_s,\alpha_s)(X_s',\alpha_s')L_s'\right]\d W_s,\nonumber\\
D_{\alpha}&J(X,\alpha)(K)=\E\left[\int_0^T\langle \nabla_uf(t,X_t,\rho_t,\alpha_t),K_t\rangle\d t+\int_0^T\langle\pa_{\nu}f(t,X_t,\rho_t,\alpha_t)(X_t',\alpha_t'),K_t'\rangle\d t\right],\\
D_{\alpha}&\Phi_i(X,\alpha)(K)(t)=-\E\left[\langle D_{\gamma}\phi^i(t,X_{t},\mu_t,\widetilde{\alpha}^t),\widetilde{K}^t\rangle_{\X}\right],\\
D_{\alpha}&\Psi_i(X,\alpha)(K)(t)=-\langle\nabla_u\psi^j(t,X_{t},\rho_t,\alpha_t),K_t\rangle-\E'\left[\langle\pa_{\nu}\psi^j(t,X_t,\rho_t,\alpha_t)(X_t',\alpha_t'),K_t'\rangle \right],\\
D_{\alpha}&F(X,\alpha)(K)(\cdot)=-\int_0^{\cdot}\nabla_ub(s,X_s,\rho_s,\alpha_s)K_s\d s-\int_0^{\cdot}\nabla_u\sigma(s,X_s,\rho_s,\alpha_s)K_s\d W_s\nonumber\\
&\quad-\int_0^{\cdot}\E'\left[\pa_{\nu}b(s,X_s,\rho_s,\alpha_s)(X_s',\alpha_s')K_s'\right]\d s-\int_0^{\cdot}\E'\left[\pa_{\nu}\sigma(s,X_s,\rho_s,\alpha_s)(X_s',\alpha_s')K_s'\right]\d W_s.
\end{align*}
Here, we have used the notation $Mz:=(M_1z,\ldots,M_rz)\in \R^{n\times r}$ for any $M=(M_1,\ldots,M_r)\in\R^{n\times n\times r}$ and $z\in\R^n$, and $(X',\alpha',L',K')$ is  an independent copy of $(X,\alpha,L,K)$ defined on the probability space $(\Omega',\F',P')$ constructed in \lemref{Lpartialdif}.
\end{lemma}

\begin{proof}
The result is a direct consequence of straightforward calculations, \lemref{Lpartialdif} as well as the fact that $\|L\|_{\Lb_{\Fb}^2}\leq\|L\|_{\M^2}$.
\end{proof}

We next provide the well-posed lemma whose proof is delegated into Appendix~\ref{appendix}.
\begin{lemma}\label{well-posedSDE}
Let $(B,\Sigma):[0,T]\times\Omega\times\R^n\times L^2(\Omega;\R^n)\mapsto\R^n\times\R^{n\times r}$ be progressively measurable mappings and satisfy that, there exists a constant $M>0$ such that, for all $(t,\omega)\in [0,T]\times\Omega$ and $(x,X),(y,Y)\in\R^n\times L^2(\Omega;\R^n)$,
{\begin{align}\label{Lipschitz}
&|(B,\Sigma)(t,\omega,x,X)-(B,\Sigma)(t,\omega,y,Y)|\leq M\left\{|x-y|+\sqrt{\E[|X-Y|^2]}\right\},
\end{align}}
and for any $R>0$,
\begin{align}\label{L2}
\E\left[\int_0^T\sup_{|x|^2+\E[|X|^2]\leq R}\left(|B(t,\cdot,x,X)|^2+|\Sigma(t,\cdot,x,X)|^2\right)\d t\right]<+\infty.
\end{align}
Then, the following {\rm(}self-dependent{\rm)} SDE admits a unique strong solution
\begin{align}\label{selfSDE}
\d Y_t=B(t,\omega,Y_t,\overline{Y}_t)\d t+\Sigma(t,\omega,Y_t,\overline{Y}_t)\d W_t,\quad Y_0\in L^2((\Omega,\F_0,\Pb);\R^n),\quad t\in(0,T].
\end{align}
Here, for $t\in[0,T]$, $\overline{Y}_t$ stands for the r.v. $Y_t$ instead of the evaluation $Y_t(\omega)$ at $\omega\in\Omega$.
\end{lemma}
It is not difficult to see that if the solution $Y=(Y_t)_{t\in[0,T]}$ to the SDE~\equref{selfSDE} exists, it must belong to $\M^2$ due to the square-integrable condition \equref{L2}. Then, in lieu of Lemma~\ref{well-posedSDE}, we have the next result.

\begin{lemma}\label{bijective}
For any $(X,\alpha)\in\tX$, the mapping $D_XF(X,\alpha):\M^2\mapsto\M^2$ is bijective.
\end{lemma}

Recall the closed and convex cones $K_1$ and $K_2$ introduced in \eqref{eq:K12}. It follows from the Riesz representation theorem that 
\begin{align}\label{eq:K12+}
K_1^+=\Mc^+[0,T],\quad K_2^+=\{f\in\Lb^{2}_{\Fb,\R};~f(t,\omega)\geq 0,~\forall(t,\omega)\in[0,T]\times\Omega\}
=\Lb_{\Fb,\R+}^2.
\end{align}
The next result follows from \thmref{fj_condition} together with \eqref{eq:XYHexam} and \eqref{eq:K12+}.
\begin{lemma}[FJ optimality condition for problem \eqref{optimization}]\label{lem:FJCondAbstractprob}
Let $(\widehat{X},\widehat{\alpha})\in C=\M^2\times\U[0,T]$ be an optimal pair to the C-MFC problem \eqref{cost_func} {\rm(}equivalent to the abstract formulation \eqref{optimization}{\rm)}. Then,  there exist (nontrivial) multipliers $r_0\geq0$, $\mu^i\in \Mc^+[0,T]$ for $i\in I_1$, $\eta^j\in\Lb^{2}_{\Fb,\R+}$ for $j\in I_2$ and $\lambda\in\M^2$ which are not all zeros such that
\begin{align}\label{complementary_slackness1} 
\begin{cases}
	\displaystyle    \int_{[0,T]}\Phi_i(\widehat{X},\widehat{\alpha})(t)\mu^i(\d t)=0,~\forall i\in I_1;\quad \E\left[\int_0^T\Psi_j(\widehat{X},\widehat{\alpha})(t)\eta^j_t\d t\right]=0,~\forall j\in I_2,\\[1.4em]
	\displaystyle r_0+\sum_{i\in I_1}\|\mu^i\|_{\rm TV}+\sum_{j\in I_2}\|\eta^j\|_{\Lb^{2}_{\Fb,\R}}+\|\lambda\|_{\M^2}=1,\\[1.4em]
	\displaystyle -\xi_1= r_0 D_XJ(\widehat{X},\widehat{\alpha})(\cdot)+\sum_{i\in I_1}\int_{[0,T]}D_X\Phi_i(\widehat{X},\widehat{\alpha})(\cdot)(t)\mu^i(\d t)\\[0.8em]
	\displaystyle \quad\qquad+\sum_{j\in I_2}\E\left[\int_0^TD_X\Psi_j(\widehat{X},\widehat{\alpha})(\cdot)(t)\eta^j_t\d t\right]-\left\langle \lambda, D_XF(\widehat{X},\widehat{\alpha})(\cdot)\right\rangle_{\M^2}~~\text{\rm in}~\M^2,\\[1.4em]
	\displaystyle -\xi_2= r_0 D_{\alpha}J(\widehat{X},\widehat{\alpha})(\cdot)+\sum_{i\in I_1}\int_{[0,T]}D_{\alpha}\Phi_i(\widehat{X},\widehat{\alpha})(\cdot)(t)\mu^i(\d t)\\[0.8em]
	\displaystyle \quad\qquad+\sum_{j\in I_2}\E\left[\int_0^TD_{\alpha}\Psi_j(\widehat{X},\widehat{\alpha})(\cdot)(t)\eta^j_t\d t\right]-\left\langle \lambda, D_{\alpha}F(\widehat{X},\widehat{\alpha})(\cdot)\right\rangle_{\M^2}~~\text{\rm in}~\Lb_{\Fb}^2,
\end{cases}
\end{align}
where $\|\cdot\|_{\rm TV}$ denotes the total variation norm on the space of all {\rm(}signed{\rm)} Radon measures and the element $\xi=(\xi_1,\xi_2)\in N_C(\widehat{X},\widehat{\alpha})\subset\tX=\M^2\times\Lb_{\Fb}^2$. 
\end{lemma}

Building upon the FJ optimality condition for the abstract problem \eqref{optimization}, we can obtain the stochastic maximum principle (SMP) for our C-MFC problems, which is another main result of this paper.

\begin{theorem}[SMP for C-MFC problems]\label{SMP_con} Let $(\widehat{X},\widehat{\alpha})\in C:=\M^2\times\U[0,T]$ be an optimal solution to our C-MFC problem \eqref{cost_func}. There exist some $r_0\geq 0$, $\mu^i\in\Mc^+[0,T]$ for $i\in I_1$ and $\eta^j\in \Lb^2_{\Fb,\R+}$ for $j\in I_2$ such that 
\begin{itemize}
\item [{\rm(i)}] the Radon measure $\mu^i$ is supported on $\{t\in[0,T];~\E[\phi^i(t,\widehat{X}_t,\widehat{\mu}_t,\widetilde{\widehat{\alpha}}^t)]=0\}$ for all $i\in I_1$,
\item [{\rm(ii)}] for $j\in I_2$, since $\eta_t^j\geq 0$, $\d t\times\d \Pb$-$\as$ and $\eta^j\in\Lb^{2}_{\Fb,\R+}$, we can define the {\rm(}continuous{\rm)} increasing process $A^j=(A_t^j)_{t\in[0,T]}$ by $A_t^j=\int_0^t\eta_s^j\d s$. Then, for $\Pb$-$\as$ $\omega$, the process $A^j(\omega)$ only increases on the set $\{t\in[0,T];~\psi^j(t,\widehat{X}_t(\omega),\widehat{\rho}_t,\widehat{\alpha}_t(\omega))=0\}$, $\d t$-$\as$,
\item [{\rm(iii)}] it holds that $\displaystyle r_0+\sum_{i\in I_1}\mu^i([0,T])+\sum_{j\in I_2}\|\eta^j\|_{\Lb^{2}_{\Fb,\R}}=1$.
\end{itemize}
Furthermore, the stochastic {\rm(}first-order{\rm)} minimum condition holds that, for any $u\in U$ and $\d t\times\d \Pb$-$\as$, 
\begin{align}\label{SMP}
&\left\langle \nabla_u H^{r_0}(t,\widehat{X}_t,\widehat{\rho}_t,\widehat{\alpha}_t,Y_t,Z_t)+\E'\left[\pa_{\nu}H^{r_0}(t,\widehat{X}_t',\widehat{\rho}_t,\widehat{\alpha}_t',Y_t,Z_t)(\widehat{X}_t,\widehat{\alpha}_t)\right]\right.\nonumber\\
&\qquad-\left.\sum_{i\in I_1}\int_{[t,T]}D_{\gamma}\phi^i(s,\widehat{X}_s,\widehat{\mu}_s,\widetilde{\widehat{\alpha}}^s)_t\mu^i(\d s)-\sum_{j\in I_2}\nabla_u\psi^j(t,\widehat{X}_t,\widehat{\rho}_t,\widehat{\alpha}_t)\eta_t^j\right.\nonumber\\
&\qquad-\left.\sum_{j\in J_2}\E'\left[\pa_{\nu}\psi^j(t,\widehat{X}_t',\widehat{\rho}_t,\widehat{\alpha}_t')(\widehat{X}_t,\widehat{\alpha}_t){\eta_t^j}'\right], u-\widehat{\alpha}_t\right\rangle\geq 0,~ \forall u\in U.
\end{align}
Here,  $D_{\gamma}\phi^i(s,\widehat{X}_s,\widehat{\mu}_s,\widetilde{\widehat{\alpha}}^s)_t$ is the evaluation of the function $D_{\gamma}\phi^i(s,\widehat{X}_s,\widehat{\mu}_s,\widetilde{\widehat{\alpha}}^s)$ at time $t$, $\widehat{\mu}_t:=\Law(\widehat{X}_t)$ and $\widehat{\rho}_t:={\cal L}(\widehat{X}_t,\widehat{\alpha}_t)$ for $t\in[0,T]$. The Hamiltonian $H^{r_0}:[0,T]\times\R^n\times\Pc_2(\R^n\times\R^l)\times\R^l\times\R^n\times\R^{n\times r}\mapsto \R$ is defined  by
\begin{align}\label{Hamiltonian}
H^{r_0}(t,x,\rho,u,y,z):=\langle b(t,x,\rho,u),y\rangle+\tr\left(\sigma(t,x,\rho,u)z^{\T}\right)+r_0f(t,x,\rho,u),
\end{align}
and the process pair $(Y,Z)=(Y_t,Z_t)_{t\in[0,T]}$ taking values in $\R^n\times\R^{n\times r}$ is the unique solution to the following general type of constrained BSDE:
\begin{align}\label{BSDE}
Y_t&=r_0\nabla_xg(\widehat{X}_T,\widehat{\mu}_T)+r_0\E'\left[\pa_{\mu}g(\widehat{X}_T',\widehat{\mu}_T)(\widehat{X}_T)\right]+\int_t^T\nabla_x H^{r_0}(s,\widehat{X}_s,\widehat{\rho}_s,\widehat{\alpha}_s,Y_s,Z_s)\d s\nonumber\\
&\quad+\int_t^T\E'\left[\pa_{\mu}H^{r_0}(s,\widehat{X}_s',\widehat{\rho}_s,\widehat{\alpha}_s',Y_s,Z_s)(\widehat{X}_s,\widehat{\alpha}_s)\right]\d s\\
&\quad-\sum_{i\in I_1}\int_{[t,T]}\left(\nabla_x\phi^i(s,\widehat{X}_s,\widehat{\mu}_s,\widetilde{\widehat{\alpha}}^s)+\E'\left[\pa_{\mu}\phi^i(s,\widehat{X}_s',\widehat{\mu}_s,\widetilde{\widehat{\alpha}}^{s'})(\widehat{X}_s) \right]\right)\mu^i(\d s)\nonumber\\
&\quad-\sum_{j\in I_2}\int_t^T\left(\nabla_x\psi^j(s,\widehat{X}_s,\widehat{\rho}_s,\widehat{\alpha}_s)\eta_s^j+\E'\left[\pa_{\mu}\psi^j(s,\widehat{X}_s',\widehat{\rho}_s,\widehat{\alpha}_s')(\widehat{X}_s,\widehat{\alpha}_s){\eta_s^j}'\right]\right)\d s-\int_t^TZ_s\d W_s.\nonumber
\end{align}
Here, $(\widehat{X}',\widehat{\alpha}')$ is  an independent copy of $(\widehat{X},\widehat{\alpha})$ defined on the probability space $(\Omega',\F',P')$ constructed following \lemref{Lpartialdif}.
\end{theorem}

\begin{remark}\label{rem:classicalSMP}
In terms of \eqref{Hamiltonian}, when $r_0=1$, our Hamiltonian $H^{r_0}$ reduces to the classical Hamiltonian. In particular, when $m=0$ {\rm(}i.e., there are no inequality constraints in problem \eqref{cost_func}{\rm)}, we have $r_0=1$ by using \corref{CQcon}, \lemref{bijective} (or simply \thmref{SMP_con}-(iii)), and hence \thmref{SMP_con} reduces to the classical stochastic maximum principle of the extended MFC problem (c.f. Theorem 3.2 of Acciaio et al. \cite{Carmona1}). 
\end{remark}

{\begin{remark}
Recently, Guo et al.~\cite{Guoxin} and Hu and Lyu \cite{Hu25} proposed a  distinct primal-dual formulation for MFG problems, where their primal optimization variable is chosen as the occupation measure. Within this framework, the Fokker-Planck equation is imposed as a linear equality constraint on occupation measure and the feasible dual variables correspond to supersolutions to the associated HJB equation. By establishing the strong duality under the regularity assumptions, the dual optimizer is the solution to their HJB equation. In contrast, our primal optimization for the C-MFC operates at the level of stochastic processes. Indeed, we choose both the state and control processes $(X, \alpha)$ as the optimization variables and treat the McKean-Vlasov SDE as an infinite-dimensional equality constraint. The dual variables (the Lagrangian multipliers) in our formulation correspond to the adjoint processes $(Y,Z)$ in the SMP (see \equref{SMP}) as the solution to the constrained BSDE in \equref{BSDE}. 

Although these two primal-dual approaches differ fundamentally, we conjecture an intrinsic connection between two (optimal) dual variables in the sense that the Lions derivative of the dual optimizer (as solution to the HJB equation) evaluated along the optimal state process coincides with our generalized Lagrangian multiplier, i.e., the solution to the constrained BSDE. That is, we conjecture that two different primal-dual formulations are connected via the relationship between the HJB equation and the SMP in the context of C-MFC problems, which has not been explored in the literature. We leave this interesting open problem for the future study.

\end{remark}}

\begin{proof}[Proof of \thmref{SMP_con}]
By using the definition \eqref{eq:NCX} of normal cone $N_C(\widehat{X},\widehat{\alpha})$, it follows that, for all $(X,\alpha)\in C:=\M^2\times\U[0,T]$ and $\xi=(\xi_1,\xi_2)\in N_C(\widehat{X},\widehat{\alpha})$,
\begin{align*}
\langle \xi_1,\widehat{X}-X\rangle_{\M^2}+\langle\xi_2,\widehat{\alpha}-\alpha\rangle_{\Lb_{\Fb}^2}\geq 0.
\end{align*}
Taking $\alpha=\widehat{\alpha}$. Then, we can conclude that $\xi_1=0$, and hence $\xi_2\in N_{\U[0,T]}(\widehat{\alpha})$, i.e., 
\begin{align}\label{optimal_condition}
\langle\xi_2,\widehat{\alpha}-\alpha\rangle_{\Lb_{\Fb}^2}\geq 0,~~\forall\alpha\in\U[0,T].
\end{align}
We can thus rewrite the fourth equality documented in \eqref{complementary_slackness1} of \lemref{lem:FJCondAbstractprob} as follows:
\begin{align}\label{XFJ}
0 &= r_0 D_XJ(\widehat{X},\widehat{\alpha})(\cdot)+\sum_{i\in I_1}\int_{[0,T]}D_X\Phi_i(\widehat{X},\widehat{\alpha})(\cdot)(t)\mu^i(\d t)\nonumber\\
&\quad+\sum_{j\in I_2}\E\left[\int_0^TD_X\Psi_j(\widehat{X},\widehat{\alpha})(\cdot)(t)\eta^j_t\d t\right]-\left\langle \lambda, D_XF(\widehat{X},\widehat{\alpha})(\cdot)\right\rangle_{\M^2}.
\end{align}
We claim that $r_0$, $\mu^i$ for $i\in I_1$ and $\eta^j$ for $j\in I_2$ can not be zeros simultaneously,  since otherwise, the equality \equref{XFJ} yields that, for all $L\in\M^2$ with the nonzero multiplier $\lambda\in\M^2$,  
\begin{align*}
\left\langle \lambda, D_XF(\widehat{X},\widehat{\alpha})(L)\right\rangle_{\M^2}=0.
\end{align*}
Then, it follows from \lemref{bijective} that $\lambda=0$, which creates a contradiction.	

Now, we can make the scaling on both sides of the equality \equref{XFJ} by multiplying a positive constant $r>0$ such that $r(|r_0|+\sum_{i\in I_1}\|\mu^i\|_{\rm TV}+\sum_{j\in I_2}\|\eta^j\|_{\Lb^{2}_{\Fb,\R}})=1$. To ease the notation, we still denote by $r_0$, $\mu^i$ for $i\in I_1$, $\eta^j$ for $j\in I_2$ and $\lambda$ the scaled multipliers respectively. That is to say
\begin{align*}
|r_0|+\sum_{i\in I_1}\|\mu^i\|_{\rm TV}+\sum_{j\in I_2}\|\eta^j\|_{\Lb^{2}_{\Fb,\R}}=1.
\end{align*}

For any $L\in\M^2$, by inserting the corresponding Fr\^echet derivatives given in \lemref{Frechet_dif} into \equref{XFJ}, and let the It\^o's representation of $(\lambda,L)$ be given by  $(\lambda_t,L_t)=(\lambda_0,L_0)+\int_0^t(\lambda_s^1,L_s^1)\d s+\int_0^t(\lambda_s^2,L_s^2)\d W_s$ with $t\in[0,T]$ since $\lambda,L\in\M^2$, we arrive at
{\small
\begin{align}\label{Xoptimality}
	0&=r_0\E\left[\int_0^T\langle \nabla_xf(t,\widehat{X}_t,\widehat{\rho}_t,\widehat{\alpha}_t),L_t\rangle\d t+\langle \nabla_xg(\widehat{X}_T,\widehat{\mu}_T),L_T\rangle\right.\nonumber\\
	&\quad+\left.\E'\left[\int_0^T\langle\pa_{\mu}f(t,\widehat{X}_t,\widehat{\rho}_t,\widehat{\alpha}_t)(\widehat{X}_t',\widehat{\alpha}_t'),L_t'\rangle\d t+\langle\pa_{\mu}g(\widehat{X}_T,\widehat{\mu}_T)(\widehat{X}_T'),L_T'\rangle\right]\right]\nonumber\\
	&\quad-\sum_{i\in I_1}\int_{[0,T]}\E\left[\left\langle\nabla_x\phi^i(t,\widehat{X}_t,\widehat{\mu}_t,\widetilde{\widehat{\alpha}}^t),L_t\right\rangle+\left\langle\E'\left[\pa_{\mu}\phi^i(t,\widehat{X}_t,\widehat{\mu}_t,\widetilde{\widehat{\alpha}}^t)(\widehat{X}_t')\right],L_t'\right\rangle\right]\mu^i(\d t)\nonumber\\
	&\quad-\sum_{j\in I_2}\E\left[\int_0^T\left\langle\nabla_x\psi^j(t,\widehat{X}_t,\widehat{\rho}_t,\widehat{\alpha}_t),L_t\right\rangle+\left\langle\E'\left[\pa_{\mu}\psi^j(t,\widehat{X}_t,\widehat{\rho}_t,\widehat{\alpha}_t)(\widehat{X}_t',\widehat{\alpha}_t')\right],L_t'\right\rangle\eta^j_t\d t\right]\\
	&\quad-\E\left[\langle\lambda_0,L_0\rangle+\int_0^T\langle\lambda_t^1,L_t^1-\nabla_xb(t,\widehat{X}_t,\widehat{\rho}_t,\widehat{\alpha}_t)L_t\rangle\d t+\int_0^T\tr\left(\lambda_t^2(L_t^2-\nabla_x\sigma(t,\widehat{X}_t,\widehat{\rho}_t,\widehat{\alpha}_t)L_t)^{\T}\right)\d t \right.\nonumber\\
	&\quad+\left.\int_0^{T}\left\langle\lambda_t^1,\E'\left[\pa_{\mu}b(t,\widehat{X}_t,\widehat{\rho}_t,\widehat{\alpha}_t)(\widehat{X}_t',\widehat{\alpha}_t')L_t'\right]\right\rangle\d t+\int_0^{T}\tr\left(\lambda_t^2\left(\E'\left[\pa_{\mu}\sigma(t,\widehat{X}_t,\widehat{\rho}_t,\widehat{\alpha}_t)(\widehat{X}_t',\widehat{\alpha}_t')L_t'\right]\right)^{\T}\right)\d t\right]\nonumber
	\end{align}}with  $\widehat{\rho}_t=\Law(\widehat{X}_t,\widehat{\alpha}_t)$ and $\widehat{\mu}_t=\Law(\widehat{X}_t)$ for $t\in[0,T]$ .
	We then define an element $p=(p_t)_{t\in[0,T]}$ in $\M^2$ as follows, for $t\in[0,T]$,		
	\begin{align}\label{adjointp}
p_t&=r_0\nabla_xg(\widehat{X}_T,\widehat{\mu}_T)+r_0\E'\left[\pa_{\mu}g(\widehat{X}_T',\widehat{\mu}_T)(\widehat{X}_T)\right]\nonumber\\
&\quad-\sum_{i\in I_1}\int_{[t,T]}\left(\nabla_x\phi^i(s,\widehat{X}_s,\widehat{\mu}_s,\widetilde{\widehat{\alpha}}^{s})+\E'\left[\pa_{\mu}\phi^i(s,\widehat{X}_s',\widehat{\mu}_s,\widetilde{\widehat{\alpha}}^{s'})(\widehat{X}_s)\right]\right)\mu^i(\d s)\nonumber\\
&\quad-\sum_{j\in I_2}\int_t^T\left(\nabla_x\psi^j(s,\widehat{X}_s,\widehat{\rho}_s,\widehat{\alpha}_s)\eta_s^j+\E'\left[\pa_{\mu}\psi^j(s,\widehat{X}_s',\widehat{\rho}_s,\widehat{\alpha}_s')(\widehat{X}_s,\widehat{\alpha}_s){\eta^j_s}'\right]\right)\d s\nonumber\\
&\quad+\int_t^T\left[r_0\nabla_xf(s,\widehat{X}_s,\widehat{\rho}_s,\widehat{\alpha}_s)+\nabla_xb(s,\widehat{X}_s,\widehat{\rho}_s,\widehat{\alpha}_s)\lambda_s^1+\nabla_x\sigma(s,\widehat{X}_s,\widehat{\rho}_s,\widehat{\alpha}_s)\lambda_s^2 \right]\d s\nonumber\\
&\quad+\int_t^T\E'\left[r_0\pa_{\mu}f(s,\widehat{X}_s',\widehat{\rho}_s,\widehat{\alpha}_s')(\widehat{X}_s,\widehat{\alpha}_s)+\pa_{\mu}b(s,\widehat{X}_s',\widehat{\rho}_s,\widehat{\alpha}_s')(\widehat{X}_s,\widehat{\alpha}_s){\lambda_s^1}'\right.\nonumber\\
&\qquad\qquad+\left.\pa_{\mu}\sigma(s,\widehat{X}_s',\widehat{\rho}_s,\widehat{\alpha}_s')(\widehat{X}_s,\widehat{\alpha}_s){\lambda_s^2}' \right]\d s-\int_t^T\lambda_s^2\d W_s,
\end{align}
where $MN:=\sum_{i=1}^rM_iL_i\in \R^n$ for any $M=(M_1,\ldots,M_r)\in\R^{n\times n\times r}$ and $L=(L_1,\ldots,L_r)\in\R^{n\times r}$.
Here, the integral $\int_{[t,T]}\nabla_x\phi^i(s,\widehat{X}_s,\widehat{\mu}_s,\widetilde{\widehat{\alpha}}^s)\mu^i(\d s)$ is well-defined because $\nabla_x\phi^i(s,\widehat{X}_s,\widehat{\mu}_s,\widetilde{\widehat{\alpha}}^s)$ is Borel measurable w.r.t. time $s$ and $\int_{[t,T]}\nabla_x\phi^i(s,\widehat{X}_s,\widehat{\mu}_s,\widetilde{\widehat{\alpha}}^s)\mu^i(\d s)$ is finite $\Pb$-$\as$, which results from $\E[\int_{[t,T]}|\nabla_x\phi^i(s,\widehat{X}_s,\widehat{\mu}_s,\widetilde{\widehat{\alpha}}^s)|\mu^i(\d s)]<\infty$ for $i\in I_1$.
In addition,  $\E[\int_0^T\nabla_x\psi^j(t,\widehat{X}_t,\widehat{\rho}_t,\widehat{\alpha}_t)\eta_t^j\d t]<\infty$, and hence $\int_t^T\nabla_x\psi^j(s,\widehat{X}_s,\widehat{\rho}_s,\widehat{\alpha}_s)\eta_s^j\d s$ is $\d t\times\d\Pb$-$\as$ well-defined.
Applying It\^{o}'s formula, we can derive that
\begin{align*}
&\E\left[\int_0^T\langle\lambda_t^1, L_t^1\rangle\d t\right]=\E\left[\int_0^T\langle p_t,\d L_t\rangle+\int_0^T\langle \lambda_t^1-p_t,\d L_t\rangle\right]\\
&=\E\left[\int_0^T\langle \lambda_t^1-p_t,\d L_t\rangle+\langle p_T,L_T\rangle-\langle p_0,L_0\rangle-\int_0^T\langle\d p_t,L_t\rangle-\langle p,L\rangle_T\right]\\
&=\E\left[\int_0^T\langle\lambda_t^1-p_t,\d L_t\rangle\right]+\E\left[\left\langle r_0\nabla_xg(\widehat{X}_T,\widehat{\mu}_T),L_T\right\rangle+ \left\langle r_0\E'\left[\pa_{\mu}g(\widehat{X}_T',\widehat{\mu}_T)(\widehat{X}_T)\right],L_T\right\rangle\right]\\
&\quad+\E\left[\int_0^T\left\langle r_0\nabla_xf(t,\widehat{X}_t,\widehat{\rho}_t,\widehat{\alpha}_t),L_t\right\rangle\d t+\int_0^T\left\langle r_0\E'\left[\pa_{\mu}f(t,\widehat{X}_t',\widehat{\rho}_t,\widehat{\alpha}_t')(\widehat{X}_t,\widehat{\alpha}_t)\right],L_t\right\rangle\d t \right]\\
&\quad+\E\left[\int_0^T\left(\langle\lambda_t^1,\nabla_xb(t,\widehat{X}_t,\widehat{\rho}_t,\widehat{\alpha}_t)L_t\rangle+\tr\left(\lambda_t^2\left(\nabla_x\sigma(t,\widehat{X}_t,\widehat{\rho}_t,\widehat{\alpha}_t)L_t-L_t^2\right)^{\T}\right)\right)\right]\\
&\quad+\E\left[\int_0^{T}\left\langle\lambda_t^1,\E'\left[\pa_{\mu}b(t,\widehat{X}_t',\widehat{\rho}_t,\widehat{\alpha}_t')(\widehat{X}_t,\widehat{\alpha}_t)L_t\right]\right\rangle\d t-\langle p_0,L_0\rangle\right].\\
&\quad+\E\left[\int_0^{T}\tr\left(\lambda_t^2\left(\E'\left[\pa_{\mu}\sigma(t,\widehat{X}_t',\widehat{\rho}_t,\widehat{\alpha}_t')(\widehat{X}_t,\widehat{\alpha}_t)L_t\right]\right)^{\T}\right)\d t\right]\\
&\quad-\sum_{i\in I_1}\int_{[0,T]}\E\left[\left\langle\nabla_x\phi^i(t,\widehat{X}_t,\widehat{\mu}_t,\widetilde{\widehat{\alpha}}^t)+\E'\left[\pa_{\mu}\phi^i(t,\widehat{X}_t',\widehat{\mu}_t,\widehat{\widetilde{\alpha}}^{t'})(\widehat{X}_t)\right],L_t\right\rangle\right]\mu^i(\d t)\\
&\quad-\sum_{j\in I_2}\int_0^T\E\left[\left\langle\nabla_x\psi^j(t,\widehat{X}_t,\widehat{\rho}_t,\widehat{\alpha}_t)\eta_t^j+\E'\left[\pa_{\mu}\psi^j(t,\widehat{X}_t',\widehat{\rho}_t,\widehat{\alpha}_t')(\widehat{X}_t,\widehat{\alpha}_t){\eta^j_s}'\right],L_t\right\rangle\right]\d t\\
&=\E\left[\int_0^T\langle\lambda_t^1-p_t,\d L_t\rangle\right]+\E\left[\left\langle r_0\nabla_xg(\widehat{X}_T,\widehat{\mu}_T),L_T\right\rangle+ \E'\left[\left\langle r_0\pa_{\mu}g(\widehat{X}_T,\widehat{\mu}_T)(\widehat{X}_T'),L_T'\right\rangle\right]\right]\\
&\quad+\E\left[\int_0^T\left\langle r_0\nabla_xf(t,\widehat{X}_t,\widehat{\rho}_t,\widehat{\alpha}_t),L_t\right\rangle\d t+\int_0^T\left\langle r_0\E'\left[\pa_{\mu}f(t,\widehat{X}_t,\widehat{\rho}_t,\widehat{\alpha}_t)(\widehat{X}_t',\widehat{\alpha}_t')\right],L_t'\right\rangle\d t \right]\\
&\quad+\E\left[\int_0^T\left(\langle\lambda_t^1,\nabla_xb(t,\widehat{X}_t,\widehat{\rho}_t,\widehat{\alpha}_t)L_t\rangle+\tr\left(\lambda_t^2\left(\nabla_x\sigma(t,\widehat{X}_t,\widehat{\rho}_t,\widehat{\alpha}_t)L_t-L_t^2\right)^{\T}\right)\right)\right]\\
&\quad+\E\left[\int_0^{T}\left\langle\lambda_t^1,\E'\left[\pa_{\mu}b(t,\widehat{X}_t,\widehat{\rho}_t,\widehat{\alpha}_t)(\widehat{X}_t',\widehat{\alpha}_t')L_t'\right]\right\rangle\d t-\langle p_0,L_0\rangle\right].\\
&\quad+\E\left[\int_0^{T}\tr\left(\lambda_t^2\left(\E'\left[\pa_{\mu}\sigma(t,\widehat{X}_t,\widehat{\rho}_t,\widehat{\alpha}_t)(\widehat{X}_t',\widehat{\alpha}_t')L_t'\right]\right)^{\T}\right)\d t\right]\\
&\quad-\sum_{i\in I_1}\int_{[0,T]}\E\left[\left\langle\nabla_x\phi^i(t,\widehat{X}_t,\widehat{\mu}_t,\widetilde{\widehat{\alpha}}^t),L_t\right\rangle+\left\langle\E'\left[\pa_{\mu}\phi^i(t,\widehat{X}_t,\widehat{\mu}_t,\widehat{\widetilde{\alpha}}^t)(\widehat{X}_t')\right],L_t'\right\rangle\right]\mu^i(\d t)\\
&\quad-\sum_{j\in I_2}\int_0^T\E\left[\left\langle\nabla_x\psi^j(t,\widehat{X}_t,\widehat{\rho}_t,\widehat{\alpha}_t),L_t\right\rangle+\left\langle\E'\left[\pa_{\mu}\psi^j(t,\widehat{X}_t,\widehat{\rho}_t,\widehat{\alpha}_t)(\widehat{X}_t',\widehat{\alpha}_t')\right],L_t'\right\rangle\right]\eta_t^j\d t,
\end{align*}
where $\langle p,L\rangle_T$ stands for the quadratic variation of $p$ and $L$, and we have applied the Fubini theorem. Plugging the last equality back into \equref{Xoptimality}, we can conclude that
\begin{align*}
\E\left[\langle p_0-\lambda_0, L_0\rangle+\int_0^T\langle p_t-\lambda_t^1,\d L_t\rangle \right]=0,~~\forall L\in\M^2.   
\end{align*}
By the arbitrariness of the element $L\in\M^2$ and the continuity of $t\mapsto p_t$, we can derive that $p_0=\lambda_0$, $\Pb$-$\as$, and $p$ is indistinguishable from $\lambda^1=(\lambda_t^1)_{t\in[0,T]}$. Together with \equref{adjointp}, this results in that, for $t\in[0,T]$,
\begin{align}\label{adjoint}
\lambda_t^1&=r_0\nabla_xg(\widehat{X}_T,\widehat{\mu}_T)+r_0\E'\left[\pa_{\mu}g(\widehat{X}_T',\widehat{\mu}_T)(\widehat{X}_T)\right]\nonumber\\
&\quad-\sum_{i\in I_1}\int_{[t,T]}\left(\nabla_x\phi^i(s,\widehat{X}_s,\widehat{\mu}_s,\widetilde{\widehat{\alpha}}^s)+\E'\left[\pa_{\mu}\phi^i(s,\widehat{X}_s',\widehat{\mu}_s,\widetilde{\widehat{\alpha}}^{s'})(\widehat{X}_s)\right]\right)\mu^i(\d s)\nonumber\\
&\quad-\sum_{j\in I_2}\int_t^T\left(\nabla_x\psi^j(s,\widehat{X}_s,\widehat{\rho}_s,\widehat{\alpha}_s)\eta_s^j+\E'\left[\pa_{\mu}\psi^j(s,\widehat{X}_s',\widehat{\rho}_s,\widehat{\alpha}_s')(\widehat{X}_s,\widehat{\alpha}_s){\eta_s^j}'\right]\right)\d s\nonumber\\
&\quad+\int_t^T\left[r_0\nabla_xf(s,\widehat{X}_s,\widehat{\rho}_s,\widehat{\alpha}_s)+\nabla_xb(s,\widehat{X}_s,\widehat{\rho}_s,\widehat{\alpha}_s)\lambda_s^1+\nabla_x\sigma(s,\widehat{X}_s,\widehat{\rho}_s,\widehat{\alpha}_s)\lambda_s^2 \right]\d s\nonumber\\
&\quad+\int_t^T\E'\left[r_0\pa_{\mu}f(s,\widehat{X}_s',\widehat{\rho}_s,\widehat{\alpha}_s')(\widehat{X}_s,\widehat{\alpha}_s)+\pa_{\mu}b(s,\widehat{X}_s',\widehat{\rho}_s,\widehat{\alpha}_s')(\widehat{X}_s,\widehat{\alpha}_s){\lambda_s^1}'\right.\nonumber\\
&\qquad\qquad\quad+\left.\pa_{\mu}\sigma(s,\widehat{X}_s',\widehat{\rho}_s,\widehat{\alpha}_s')(\widehat{X}_s,\widehat{\alpha}_s){\lambda_s^2}' \right]\d s-\int_t^T\lambda_s^2\d W_s.
\end{align}
Here, we recall that $(\widehat{X}',\widehat{\alpha}',\lambda^{1'},\lambda^{2'})$ is  an independent copy of $(\widehat{X},\widehat{\alpha},\lambda^{1},\lambda^{2})$ defined on the probability space $(\Omega',\F',P')$ constructed following \lemref{Lpartialdif}.

Now, if we set $Q:\M^2\mapsto\M^2$ that maps any $\eta\in\M^2$ admitting the It\^o's representation $\eta_t=\eta_0+\int_0^t\eta_s^1\d s+\int_0^t\eta_s^2\d W_s$ with $t\in[0,T]$ to $Q(\eta)(\cdot):=\eta_0+\int_0^{\cdot}\eta_s\d s+\int_0^{\cdot}\eta_s^2\d W_s$ for $t\in[0,T]$. Then, we can deduce that $Q(\lambda^1)$ is indistinguishable from $\lambda$. Moreover, combining the Fr\^echet derivative provided in \lemref{Frechet_dif} together with \equref{optimal_condition} and the fifth equality in \eqref{complementary_slackness1} on the FJ condition in \lemref{lem:FJCondAbstractprob}, we get that, for any $\alpha\in \U[0,T]$,
{\small \begin{align}\label{alphaopt}
	0&\leq\E\left[\int_0^Tr_0\left\langle \nabla_uf(t,\widehat{X}_t,\widehat{\rho}_t,\widehat{\alpha}_t)+\E'\left[\pa_{\nu}f(t,\widehat{X}_t',\widehat{\rho}_t,\widehat{\alpha}_t')(\widehat{X}_t,\widehat{\alpha}_t)\right] ,\alpha_t-\widehat{\alpha}_t\right\rangle\d t\right.\nonumber\\
	&\quad+\int_0^T\left\langle \nabla_u b(t,\widehat{X}_t,\widehat{\rho}_t,\widehat{\alpha}_t)^{\T}\lambda_t^1+\nabla_u\sigma(t,\widehat{X}_t,\widehat{\rho}_t,\widehat{\alpha}_t)^{\T}\lambda_t^2,\alpha_t-\widehat{\alpha}_t\right\rangle\d t\nonumber\\
	&\quad+\int_0^T\left\langle \E'\left[\pa_{\nu} b(t,\widehat{X}_t',\widehat{\rho}_t,\widehat{\alpha}_t')(\widehat{X}_t,\widehat{\alpha}_t)^{\T}{\lambda_t^1}'+\pa_{\nu}\sigma(t,\widehat{X}_t',\widehat{\rho}_t,\widehat{\alpha}_t')(\widehat{X}_t,\widehat{\alpha}_t)^{\T}{\lambda_t^2}'\right],\alpha_t-\widehat{\alpha}_t\right\rangle\d t\nonumber\\
	&\quad-\sum_{i\in I_1}\int_{[0,T]}\int_0^t\left\langle D_{\gamma}\phi^i(t,\widehat{X}_t,\widehat{\mu}_t,\widetilde{\widehat{\alpha}}^t)_s, \alpha_s-\widehat{\alpha}_s\right\rangle\d s\mu^i(\d t)\nonumber\\
	&\quad-\left.\sum_{j\in I_2}\int_0^T\left\langle\left(\nabla_u\psi^j(t,\widehat{X}_t,\widehat{\rho}_t,\widehat{\alpha}_t)\eta_t^j+\E'\left[\pa_{\nu}\psi^j(t,\widehat{X}_t',\widehat{\rho}_t,\widehat{\alpha}_t')(\widehat{X}_t,\widehat{\alpha}_t){\eta_t^j}'\right]\right),\alpha_t-\widehat{\alpha}_t\right\rangle\d t\right]\nonumber\\
	&=\E\left[\int_0^Tr_0\left\langle \nabla_uf(t,\widehat{X}_t,\widehat{\rho}_t,\widehat{\alpha}_t)+\E'\left[\pa_{\mu}f(t,\widehat{X}_t',\widehat{\rho}_t,\widehat{\alpha}_t')(\widehat{X}_t,\widehat{\alpha}_t)\right] ,\alpha_t-\widehat{\alpha}_t\right\rangle\d t\right.\nonumber\\
	&\quad+\int_0^T\left\langle \nabla_u b(t,\widehat{X}_t,\widehat{\rho}_t,\widehat{\alpha}_t)^{\T}\lambda_t^1+\nabla_u\sigma(t,\widehat{X}_t,\widehat{\rho}_t,\widehat{\alpha}_t)^{\T}\lambda_t^2,\alpha_t-\widehat{\alpha}_t\right\rangle\d t\nonumber\\
	&\quad+\int_0^T\left\langle \E'\left[\pa_{\nu} b(t,\widehat{X}_t',\widehat{\rho}_t,\widehat{\alpha}_t')(\widehat{X}_t,\widehat{\alpha}_t)^{\T}{\lambda_t^1}'+\pa_{\nu}\sigma(t,\widehat{X}_t',\widehat{\rho}_t,\widehat{\alpha}_t')(\widehat{X}_t,\widehat{\alpha}_t)^{\T}{\lambda_t^2}'\right],\alpha_t-\widehat{\alpha}_t\right\rangle\d t\nonumber\\
	&\quad-\sum_{i\in I_1}\int_0^T\left\langle \int_{[t,T]}D_{\gamma}\phi^i(s,\widehat{X}_s,\widehat{\mu}_s,\widetilde{\widehat{\alpha}}^s)_t\mu^i(\d s), \alpha_s-\widehat{\alpha}_s\right\rangle\d t\nonumber\\
	&\quad-\left.\sum_{j\in I_2}\int_0^T\left\langle\left(\nabla_u\psi^j(t,\widehat{X}_t,\widehat{\rho}_t,\widehat{\alpha}_t)\eta_t^j+\E'\left[\pa_{\nu}\psi^j(t,\widehat{X}_t',\widehat{\rho}_t,\widehat{\alpha}_t')(\widehat{X}_t,\widehat{\alpha}_t){\eta_t^j}'\right]\right),\alpha_t-\widehat{\alpha}_t\right\rangle\d t\right],
	\end{align}}where we used the Fubini theorem and the fact that $\int_t^TD_{\gamma}\phi^i(s,\widehat{X}_s,\widehat{\mu}_s,\widetilde{\widehat{\alpha}}^s)_t\mu^i(\d s)$ can be viewed as a Bochner integral valued in $\R^l$. Here,  $D_{\gamma}\phi^i(s,\widehat{X}_s,\widetilde{\widehat{\alpha}}^s)_t$ is the evaluation of the continuous function $D_{\gamma}\phi^i(s,\widehat{X}_s,\widetilde{\widehat{\alpha}}^s)$ at time $t$. Then, the arbitrariness of $\alpha\in\U[0,T]$ in \equref{alphaopt} yields that
	\begin{align}\label{optimal_control}
&\left\langle \E'\left[r_0\pa_{\nu}f(t,\widehat{X}_t',\widehat{\rho}_t,\widehat{\alpha}_t')(\widehat{X}_t,\widehat{\alpha}_t)+\pa_{\nu} b(t,\widehat{X}_t',\widehat{\rho}_t,\widehat{\alpha}_t)(\widehat{X}_t,\widehat{\alpha}_t)^{\T}{\lambda_t^1}'+\pa_{\nu}\sigma(t,\widehat{X}_t',\widehat{\rho}_t,\widehat{\alpha}_t')(\widehat{X}_t,\widehat{\alpha}_t)^{\T}{\lambda_t^2}'\right]\right.\nonumber\\
&\qquad+\left.r_0\nabla_uf(t,\widehat{X}_t,\widehat{\rho}_t,\widehat{\alpha}_t)+\nabla_ub(t,\widehat{X}_t,\widehat{\rho}_t,\widehat{\alpha}_t)^{\T}\lambda_t^1+\nabla_u\sigma(t,\widehat{X}_t,\widehat{\rho}_t,\widehat{\alpha}_t)^{\T}\lambda_t^2\right.\nonumber\\
&\qquad-\left.\sum_{i\in I_1}\int_{[t,T]}D_{\gamma}\phi^i(s,\widehat{X}_s,\widehat{\mu}_s,\widetilde{\widehat{\alpha}}^s)_t\mu^i(\d s)-\sum_{j\in I_2}\nabla_u\psi^j(t,\widehat{X}_t,\widehat{\rho}_t,\widehat{\alpha}_t)\eta_t^j\right.\nonumber\\
&\qquad-\left.\sum_{j\in J_2}\E'\left[\pa_{\nu}\psi^j(t,\widehat{X}_t',\widehat{\rho}_t,\widehat{\alpha}_t')(\widehat{X}_t,\widehat{\alpha}_t){\eta_t^j}'\right], u-\widehat{\alpha}_t\right\rangle\geq 0,~ \forall u\in U,\quad\d t\times\d\Pb\mbox{-}\as
\end{align}
Using the representation \eqref{Hamiltonian} of the Hamiltonian $H^{r_0}:[0,T]\times\R^n\times\Pc_2(\R^n\times \R^l)\times\R^l\times\R^n\times\R^{n\times r}\mapsto\R$, we obtain the stochastic minimum condition \equref{SMP} from \equref{optimal_control} by noting that $\lambda^1=Y$ and $\lambda^2=Z$, where the process pair $(Y,Z)=(Y_t,Z_t)_{t\in[0,T]}$ satisfies the BSDE~\eqref{BSDE}.

On the other hand, the Radon measure $\mu^i$ is supported on $\{t\in[0,T];~\E[\phi^i(t,\widehat{X}_t,\widehat{\mu}_t,\widetilde{\widehat{\alpha}}^t)]=0\}$ for $i\in I_1$, which is a consequence of the first complementary slackness condition given in \eqref{complementary_slackness1}  of \lemref{lem:FJCondAbstractprob}. Moreover, $\eta_t^j$ vanishes on the set $\{t\in[0,T];~\psi^j(t,\widehat{X}_t(\omega),\widehat{\rho}_t,\widehat{\alpha}_t(\omega))\neq0\}$, $\d t$-$\as$, which results from the second complementary slackness condition  in \eqref{complementary_slackness1}  of \lemref{lem:FJCondAbstractprob}. In addition, we have that $r_0+\sum_{i\in I_1}\mu^i([0,T])+\sum_{j\in I_2}\|\eta^j\|_{\Lb^{2}_{\Fb,\R}}=r_0+\sum_{i\in I_1}\|\mu^i\|_{\rm TV}+\sum_{j\in I_2}\|\eta^j\|_{\Lb^{2}_{\Fb,\R}}=1$ as $\mu^i\in\Mc^+[0,T]$ for $i\in I_1$. Finally, the uniqueness of solution to BSDE \equref{BSDE} is a result of the classical BSDE theory (c.f. Pham \cite{Pham}). 
\end{proof}

\begin{remark}
According to the proof of \thmref{SMP_con}, the adjoint process $(Y,Z)=(Y_t,Z_t)_{t\in [0,T]}$ in the constrained BSDE \eqref{BSDE} can be interpreted as the components of the Lagrange multiplier with respect to the McKean-Vlasov SDE {\rm(}equality{\rm)} constraint. 
\end{remark}

{
By applying \propref{fj_sufficient} to \equref{cost_func}, we can also derive sufficient optimality conditions for the C-MFC problem \equref{cost_func}, which can guarantee the optimality of any admissible pair satisfying the necessary conditions in \thmref{SMP_con} with $r_0>0$:
\begin{prop}[Sufficient SMP for C-MFC problems]\label{smp_sufficient}
We impose the condition on the coefficient $(b,\sigma)$ in the following sense, for any $\lambda_1,\lambda_2\in\R$, $(x_1,u_1),(x_2,u_2)\in\R^n\times U$ and $(X_1,\alpha_1),(X_2,\alpha_2)\in L^2((\tilde\Omega,\tilde\F,\tilde\Pb);\R^n\times\R^l)$, 
it holds that
\begin{align*}
&(b,\sigma)(t,\lambda_1x_1+\lambda_2x_2,\Law(\lambda_1X_1+\lambda_2X_2,\lambda_1\alpha_1+\lambda_2\alpha_2),\lambda_1u_1+\lambda_2u_2)\\
&\qquad=\lambda_1(b,\sigma)(t,x_1,\Law(X_1,\alpha_1),u_1)
+\lambda_2(b,\sigma)(t,x_2,\Law(X_2,\alpha_2),u_2).
\end{align*}
Assume further that $(f,g,(-\phi^i)_{i\in I},(-\psi^j)_{j\in J})$ are convex in $(x,u)\in\R^n\times U$, and are $L$-convex (c.f. Definition 3.8 in \cite{BWWX24}) in $\rho\in\Pc_2(\R^n\times\R^l)$. Then, any admissible pair $(\widehat{X},\widehat{\alpha})\in C:=\M^2\times\U[0,T]$ satisfying the necessary conditions in \thmref{SMP_con} with $r_0>0$ is optimal for the C-MFC problem \eqref{cost_func}.
\end{prop}

We stress that, when $\phi^i=0$ and $\psi^j=0$ for $(i,j)\in I\times J$, the above sufficient condition reduces to the sufficient condition of the classical SMP. Indeed, under assumptions in Proposition \ref{smp_sufficient}, the Hamiltonian is convex in $(x,u)\in\R^n\times U$, and is $L$-convex in the measure argument $\rho\in\Pc_2(\R^n\times\R^l)$, which coincides with the standard convexity condition that is usually invoked in the sufficient condition for the classical SMP (c.f.  Theorem 3.5 in Acciaio et al. \cite{Carmona1}).
}

\section{Constrained Stochastic Control and MFG Problems}\label{sec:constrainedSCP-MFG}

\subsection{Constrained stochastic control problem}\label{sec:CSCP0}

In this subsection, we consider the classical constrained stochastic control problem as a special case of MFC problem \eqref{cost_func} given by
\begin{align}\label{cost_func-withoutlaw}
\begin{cases}
\text{minimize $J(\alpha)$ over $\alpha\in\U[0,T]$;}\\[0.4em]
\displaystyle \text{subject to $X_t^{\alpha}=\kappa+\int_0^tb(s,X_s^{\alpha},\alpha_s)\d s+\int_0^t\sigma(s,X_s^{\alpha},\alpha_s)\d W_s$,}\\[0.8em]
\displaystyle ~~~~~~~~~~~~~~~\text{$\E\left[\phi^i\left(t,X_t^{\alpha},\widetilde{\alpha}^t\right)\right]\geq 0$,~~$i\in I_1$,~~$\forall t\in[0,T]$,}\\[0.6em]
\displaystyle ~~~~~~~~~~~~~~~\text{$\psi^j\left(t,X_t^{\alpha},\alpha_t\right)\geq 0$,~~$j\in I_2$,~~$\d t\times\d\Pb$-a.s..}
\end{cases}
\end{align}

Then, we have from \thmref{SMP_con} that

\begin{corollary}[SMP for constrained stochastic control problem]\label{coro:SMP_con} Let $(\widehat{X},\widehat{\alpha})\in C:=\M^2\times\U[0,T]$ be an optimal solution to the constrained stochastic control problem \eqref{cost_func-withoutlaw}. Then, there exist some $r_0\geq 0$, $\mu^i\in\Mc^+[0,T]$ for $i\in I_1$ and $\eta^j\in \Lb^2_{\Fb,\R+}$ for $j\in I_2$ such that 
\begin{itemize}
\item [{\rm(i)}] the Radon measure $\mu^i$ is supported on $\{t\in[0,T];~\E[\phi^i(t,\widehat{X}_t,\widetilde{\widehat{\alpha}}^t)]=0\}$ for all $i\in I_1$,
\item [{\rm(ii)}] for $j\in I_2$, as $\eta_t^j\geq 0$, $\d t\times\d \Pb$-$\as$ and $\eta^j\in \Lb^2_{\Fb,\R+}$, we can define the {\rm(}continuous{\rm)} increasing process $A^j=(A_t^j)_{t\in[0,T]}$ by $A_t^j=\int_0^t\eta_s^j\d s$. Then, for $\Pb$-$\as~\omega$, the process $A^j(\omega)$ only increases on the set $\{t\in[0,T];~\psi^j(t,\widehat{X}_t(\omega),\widehat{\alpha}_t(\omega))=0\}$, $\d t$-$\as$,
\item [{\rm(iii)}] it holds that $\displaystyle r_0+\sum_{i\in I_1}\mu^i([0,T])+\sum_{j\in I_2}\|\eta^j\|_{\Lb^{2}_{\Fb,\R}}=1$.
\end{itemize}
Furthermore, the stochastic {\rm(}first-order{\rm)} minimum condition holds in the sense that, for any $u\in U$ and $\d t\times\d \Pb$-$\as$, 
{\small\begin{align}\label{SMP-withoutLaw}
	\left\langle \nabla_u H^{r_0}(t,\widehat{X}_t,\widehat{\alpha}_t,Y_t,Z_t)-\sum_{i\in I_1}\int_t^TD_{\gamma}\phi^i(s,\widehat{X}_s,\widetilde{\widehat{\alpha}}^s)_t\mu^i(\d s)-\sum_{j\in I_2}\nabla_u\psi^j(t,\widehat{X}_t,\widehat{\alpha}_t)\eta_t^j, u-\widehat{\alpha}_t\right\rangle\geq 0.
	\end{align}}
	Here, the Hamiltonian $H^{r_0}:[0,T]\times\R^n\times\R^l\times\R^n\times\R^{n\times r}\mapsto \R$ is defined  by
	\begin{align}\label{Hamiltonian-withoutLaw}
H^{r_0}(t,x,u,y,z):=\langle b(t,x,u),y\rangle+\tr\left(\sigma(t,x,u)z^{\T}\right)+r_0f(t,x,u),
\end{align}
and the process pair $(Y,Z)=(Y_t,Z_t)_{t\in[0,T]}$ taking values in $\R^n\times\R^{n\times r}$ is the unique solution to the BSDE:
\begin{equation}\label{BSDE-withoutLaw}
\begin{aligned}
	Y_t&=r_0\nabla_xg(\widehat{X}_T)+\int_t^T\nabla_x H^{r_0}(s,\widehat{X}_s,\widehat{\alpha}_s,Y_s,Z_s)\d s-\sum_{i\in I_1}\int_{[t,T]}\nabla_x\phi^i(s,\widehat{X}_s,\widehat{\mu}_s,\widetilde{\widehat{\alpha}}^s)\mu^i(\d s)\\
	&\quad-\sum_{j\in I_2}\int_t^T\nabla_x\psi^j(s,\widehat{X}_s,\widehat{\rho}_s,\widehat{\alpha}_s)\d A_s^j-\int_t^TZ_s\d W_s.
\end{aligned}
\end{equation}
\end{corollary}

Consider the stochastic control problem with the dynamic expectation constraint and the dynamic state constraint specified by, for $i\in I_1$ and $j\in I_2$,
\begin{align}\label{eq:constraintsexam0}
\E[\phi^i(t,X_t^{\alpha})]\geq0,~\forall t\in[0,T],~\text{and}~ \psi^j(t,X_t^{\alpha})\geq0,~~\d t\times\d\Pb\text{-}\as
\end{align}
Note that the constraint functions $\phi^i$ and $\psi^j$ in \eqref{eq:constraintsexam0} are independent of the control process. Then, the stochastic (first-order) minimum condition \eqref{SMP-withoutLaw} is reduced to that, $\d t\times\d \Pb$-$\as$, 
\begin{align}\label{SMP-withoutLaw-withoutstate}
\left\langle \nabla_u H^{r_0}(t,\widehat{X}_t,\widehat{\alpha}_t,Y_t,Z_t), u-\widehat{\alpha}_t\right\rangle\geq 0,\quad \forall u\in U.
\end{align}
Here, the process pair $(Y,Z)=(Y_t,Z_t)_{t\in[0,T]}$ taking values in $\R^n\times\R^{n\times r}$ is the unique solution to the following BSDE:
\begin{align}\label{BSDE-withoutLaw-withoutstate}
Y_t&=r_0\nabla_xg(\widehat{X}_T)+\int_t^T\nabla_x H^{r_0}(s,\widehat{X}_s,\widehat{\alpha}_s,Y_s,Z_s)\d s-\sum_{i\in I_1}\int_{[t,T]}\nabla_x\phi^i(s,\widehat{X}_s)\mu^i(\d s)\nonumber\\
&\quad-\sum_{j\in I_2}\int_t^T\nabla_x\psi^j(s,\widehat{X}_s)\d A_s^j-\int_t^TZ_s\d W_s.
\end{align}	
In particular,	Hu et al. \cite{Hu}  considered the optimal control of SDEs with the dynamic expectation constraints described by the 1st one in \eqref{eq:constraintsexam0} (i.e., $m=1$ and $\phi:=\phi^1$). In this case, the stochastic (first-order) minimum condition \eqref{SMP-withoutLaw-withoutstate} holds true. Correspondingly, the BSDE \eqref{BSDE-withoutLaw-withoutstate} turns to be
\begin{align*}
Y_t&=r_0\nabla_xg(\widehat{X}_T)+\int_t^T\nabla_x H^{r_0}(s,\widehat{X}_s,\widehat{\alpha}_s,Y_s,Z_s)\d s-\int_{[t,T]}\nabla_x\phi(s,\widehat{X}_s)\mu^i(\d s)-\int_t^TZ_s\d W_s,
\end{align*}
which coincides with the result given in Hu et al. \cite{Hu}.

We can also consider the special case when $m=1$, $m_1=0$ and $\psi(t,x)=\psi^1(t,x)=x-a_t$ for $(t,x)\in[0,T]\times\R$. Here, $t\mapsto a_t$ is a deterministic continuous real-valued mapping. Thus, the constrained stochastic control problem \equref{cost_func-withoutlaw} reduces to \begin{align}\label{nonnegative}
\begin{cases}
\text{minimize $J(\alpha)$ over $\alpha\in\U[0,T]$;}\\[0.4em]
\displaystyle \text{subject to $X_t^{\alpha}=\kappa+\int_0^tb(s,X_s^{\alpha},\alpha_s)\d s+\int_0^t\sigma(s,X_s^{\alpha},\alpha_s)\d W_s$,}\\[0.8em]
\displaystyle ~~~~~~~~~~~~~~~\text{$X_t^{\alpha}\geq a_t$,~~$\d t\times\d\Pb$-a.s.}
\end{cases}
\end{align}
As a result of  \corref{coro:SMP_con},  
the following stochastic (first-order) minimum condition holds, $\d t\times\d \Pb$-$\as$, $\left\langle \nabla_u H^{r_0}(t,\widehat{X}_t,\widehat{\alpha}_t,Y_t,Z_t), u-\widehat{\alpha}_t\right\rangle\geq 0$ for all $u\in U$, where the process pair $(Y,Z)=(Y_t,Z_t)_{t\in[0,T]}$ satisfies the constrained BSDE:
\begin{align}\label{reflectedBSDE}
\d Y_t=-\nabla_xH^{r_0}(t,\widehat X_t,\widehat\alpha_t,Y_t,Z_t)\d t+\d A_t+Z_t\d W_t,\quad Y_T=\nabla_xg(\widehat{X}_T),
\end{align}
where $A_t=\int_0^t\eta_s\d s$ is the adapted increasing process with $\eta=(\eta_t)_{t\in[0,T]}$ being the multiplier given in \corref{coro:SMP_con} such that $\E[\int_0^T(\widehat{X}_t-a_t)\d A_t]=0$. This also implies that,  $\Pb$-$\as$, $\int_0^T(\widehat{X}_t-a_t)\d A_t=0$ 
in view of $\widehat{X}_t-a_t\geq0$, $\d t\times\d\Pb$-a.s.. Therefore, we can see that Eq.~\equref{reflectedBSDE} resembles a reflected BSDE with the reflecting barrier $t\mapsto a_t$.

\subsection{Constrained (extended) mean field game}

In this section, we discuss the constrained MFG problems (C-MFG) with the dynamic expectation constraints and/or the dynamic state-control constraints. We regard the constrained MFG problem as a constrained stochastic control problem with the fixed point consistency condition. 

We formulate the C-MFG as follows: 
\begin{align}\label{cost_func-MFG}
\begin{cases}
\text{minimize $J(\alpha)$ over $\alpha\in\U[0,T]$;}\\[0.4em]
\displaystyle \text{subject to $X_t^{\alpha}=\kappa+\int_0^tb(s,X_s^{\alpha},\rho_s,\alpha_s)\d s+\int_0^t\sigma(s,X_s^{\alpha},\rho_s,\alpha_s)\d W_s$,}\\[0.8em]
\displaystyle ~~~~~~~~~~~~~~~\text{$\E\left[\phi^i\left(t,X_t^{\alpha},\widetilde{\alpha}^t\right)\right]\geq 0$,~~$i\in I_1$,~~$\forall t\in[0,T]$,}\\[0.6em]
\displaystyle ~~~~~~~~~~~~~~~\text{$\psi^j\left(t,X_t^{\alpha},\alpha_t\right)\geq 0$,~~$j\in I_2$,~~$\d t\times\d\Pb$-a.s.}\\[0.6em]
\displaystyle ~~~~~~~~~~~~~~~\text{$\rho_t=\Law(\widehat{X}_t,\widehat{\alpha}_t)$ for $t\in[0,T]$,}
\end{cases}
\end{align}
where the measure flow $\rho=(\rho_t)_{t\in[0,T]}$ takes values in $\Pc_2(\R^n\times\R^l)$. Here, $\widehat{\alpha}=(\widehat{\alpha}_t)_{t\in[0,T]}$ is an (admissible) optimal control and $\widehat{X}=(\widehat{X}_t)_{t\in[0,T]}$ is the resulting state process. The measure flow $\rho=(\rho_t)_{t\in[0,T]}$ satisfying  $\rho_t=\Law(\widehat{X}_t,\widehat{\alpha}_t)$ for $t\in[0,T]$ is called a mean-field equilibrium (MFE) of the C-MFG problem \eqref{cost_func-MFG}.

\begin{corollary}[SMP for C-MFG problem]\label{coro:SMP_MFG} Let $\rho_t$ be an MFE to the C-MFG problem \eqref{cost_func-MFG}. Then, there exist some $r_0\geq 0$, $\mu^i\in\Mc^+[0,T]$ for $i\in I_1$ and $\eta^j\in \Lb^2_{\Fb,\R+}$ for $j\in I_2$ such that 
\begin{itemize}
\item [{\rm(i)}] the radon measure $\mu^i$ is supported on $\{t\in[0,T];~\E[\phi^i(t,\widehat{X}_t,\widetilde{\widehat{\alpha}}^t)]=0\}$ for all $i\in I_1$,
\item [{\rm(ii)}] for $j\in I_2$, as $\eta_t^j\geq 0$, $\d t\times\d \Pb$-$\as$ and $\eta^j\in \Lb^2_{\Fb,\R+}$, we can define the {\rm(}continuous{\rm)} increasing process $A^j=(A_t^j)_{t\in[0,T]}$ by $A_t^j=\int_0^t\eta_s^j\d s$. Then, for $\Pb$-$\as~\omega$, the process $A^j(\omega)$ only increases on the set $\{t\in[0,T];~\psi^j(t,\widehat{X}_t(\omega),\widehat{\alpha}_t(\omega))=0\}$, $\d t$-$\as$,
\item [{\rm(iii)}] it holds that $\displaystyle r_0+\sum_{i\in I_1}\mu^i([0,T])+\sum_{j\in I_2}\|\eta^j\|_{\Lb^{2}_{\Fb,\R}}=1$.
\end{itemize}
Furthermore, the stochastic {\rm(}first-order{\rm)} minimum condition holds in the sense that, for any $u\in U$ and $\d t\times\d \Pb$-$\as$, 
{\small	\begin{align}\label{SMP-MFG}
	\left\langle \nabla_u H^{r_0}(t,\widehat{X}_t,\widehat{\alpha}_t,\rho_t,Y_t,Z_t)-\sum_{i\in I_1}\int_t^TD_{\gamma}\phi^i(s,\widehat{X}_s,\widetilde{\widehat{\alpha}}^s)_t\mu^i(\d s)-\sum_{j\in I_2}\nabla_u\psi^j(t,\widehat{X}_t,\widehat{\alpha}_t)\eta_t^j, u-\widehat{\alpha}_t\right\rangle\geq 0,
	\end{align}}where the Hamiltonian $H^{r_0}$ is given in \eqref{Hamiltonian}, 
	and the process pair $(Y,Z)=(Y_t,Z_t)_{t\in[0,T]}$, taking values in $\R^n\times\R^{n\times r}$, is the unique solution to the BSDE:
	\begin{equation}\label{BSDE-MFG}
\begin{aligned}
	Y_t&=r_0\nabla_xg(\widehat{X}_T)+\int_t^T\nabla_x H^{r_0}(s,\widehat{X}_s,\widehat{\alpha}_s,\rho_s,Y_s,Z_s)\d s-\sum_{i\in I_1}\int_{[t,T]}\nabla_x\phi^i(s,\widehat{X}_s,\widehat{\mu}_s,\widetilde{\widehat{\alpha}}^s)\mu^i(\d s)\\
	&\quad-\sum_{j\in I_2}\int_t^T\nabla_x\psi^j(s,\widehat{X}_s,\widehat{\rho}_s,\widehat{\alpha}_s)\d A_s^j-\int_t^TZ_s\d W_s.
\end{aligned}
\end{equation}
Finally, the consistency condition holds that $\rho_t=\Law(\widehat{X}_t,\widehat{\alpha}_t)$.
\end{corollary}

\subsection{An illustrative example: constrained LQ-MFC problem}\label{sec:LQ-MFCprob}

In this subsection, we present an example of linear-quadratic MFC problem under a type of dynamic mixed state-control constraint in which the optimal control can be characterized. 

Consider the following one-dimensional controlled McKean-Vlasov SDE that, for $\alpha\in\U[0,T]$ with control space $U=\R$,
{\small\begin{align}\label{LQ}
\d X_t^{\alpha} &=\{b_1 X_t^{\alpha}+b_2\E[X_t^{\alpha}]+b_3\alpha_t+b_4\E[\alpha_t]\}\d t+\{\sigma_1 X_t^{\alpha}+\sigma_2\E[X_t^{\alpha}]+\sigma_3\alpha_t+\sigma_4\E[\alpha_t]\}\d W_t,
\end{align}} 
where $X_0=\kappa\in L^2((\Omega,\F_0,\Pb);\R)$ and $W=(W_t)_{t\in[0,T]}$ is a standard  one-dimensional Brownian motion on the probability space $(\Omega,\F,\Pb)$ with $\Fb=(\F_t)_{t\in[0,T]}$ being the natural filtration generated by $W$. Here, model parameters $b_i,\sigma_i\in\R$ for $i=1,2,3,4$. The aim of the agent is to minimize the following cost functional, for cost parameters $q,\ell\geq0$ and $v>0$,
\begin{align}\label{LQ_cost}
J(\alpha)=\frac{1}{2}\E\left[\int_0^T \{q(X_t^{\alpha})^2+v u_t^2\}\d t+\ell(X_T^{\alpha}-\E[X_T^{\alpha}])^2\right],
\end{align}
under dynamic state-law-control constraint described by, $\Pb$-a.s.
\begin{align}\label{LQ_con}
h_tX_t^{\alpha}-\alpha_t\geq 0,\quad \forall t\in[0,T]
\end{align}
with $t\mapsto h_t$ being a continuous function on $[0,T]$.

We next give a necessary optimality condition for constrained control problem \eqref{LQ}-\eqref{LQ_con} using the KKT condition that is established in \corref{CQcon}. To this end, let $(\widehat{X},\widehat{\alpha})\in C=\M^2\times\U[0,T]$ be an optimal pair. First, we verify that LICQ(($\widehat{X},\widehat{\alpha}$)) holds true. In fact, by using \lemref{Frechet_dif}, it is enough to show that the following mapping is surjective:
\begin{align*}
\left(D F(\widehat{X},\widehat{\alpha}),D\Psi(\widehat{X},\widehat{\alpha})\right):\M^2\times\Lb_{\Fb,\R}^2\to \M^2\times\Lb_{\Fb,\R}^2,
\end{align*}
where, for $(L,K)\in\M^2\times\Lb_{\Fb,\R}$,\begin{align*}
D\Psi(\widehat{X},\widehat{\alpha})(L,K)&=-h(\cdot)L_t+K_{\cdot},\\
DF(\widehat{X},\widehat{\alpha})(L,K)&=L_{\cdot}-\int_0^{\cdot}b_1L_s\d s-\int_0^{\cdot}\sigma_1L_s\d W_s-\int_0^{\cdot}b_2\E[L_s]\d s-\int_0^{\cdot}\sigma_2\E[L_s]\d W_s\\
&\quad-\int_0^{\cdot}b_3K_s\d s-\int_0^{\cdot}\sigma_3K_s\d W_s\nonumber-\int_0^{\cdot}b_4\E[K_s]\d s-\int_0^{\cdot}\sigma_4\E[K_s]\d W_s.
\end{align*}
However, for any $(L',K')\in\M^2\times\Lb_{\Fb,\R}^2$, we set $K_t=h_tL_t+K_t'$ for $t\in[0,T]$, and plug it into the above equality. Then, there exists some  $L\in \M^2$ guaranteed by \lemref{well-posedSDE} such that
\begin{align*}
\left(D\Psi(\widehat{X},\widehat{\alpha}),DF(\widehat{X},\widehat{\alpha})\right)(L,K)=(L',K').
\end{align*}
Thus, from \corref{CQcon} and \thmref{SMP_con}, it follows that the KKT condition holds. Therefore, there exists some $\eta\in\Lb_{\Fb,\R+}^2$ such that 
\begin{align}\label{minimal}
\eta_t\left(h_t\widehat{X}_t-\widehat{\alpha}_t\right)=0,\quad \d t\times\d\Pb\text{-}\as.
\end{align}
The constrained BSDE given in \thmref{SMP_con} is specified into
{\small\begin{align}\label{LQ_BSDE}
\d Y_t=-\{b_1Y_t+b_2\E[Y_t]+\sigma_1Z_t+\sigma_2\E[Z_t]+q\widehat{X}_t\}\d t+h_t\eta_t\d t+Z_t\d W_t,~Y_T=\ell\{\widehat{X}_T-\E[\widehat{X}_T]\}.
\end{align}}
Then, the stochastic minimal condition is given by, $\d t\times\d\Pb$-a.s.
\begin{align*}
\{b_3Y_t+b_4\E[Y_t]+\sigma_3Z_t+\sigma_4\E[Z_t]+v\widehat{\alpha}_t+\eta_t\}(u-\widehat{\alpha}_t)\geq 0,\quad \forall u\in U=\R,   
\end{align*}
which yields that the optimal control is characterized by  
\begin{align}\label{LQ_SMP}
\widehat{\alpha}_t=-v^{-1}\left\{b_3Y_t+b_4\E[Y_t]+\sigma_3Z_t+\sigma_4\E[Z_t]+\eta_t\right\},\quad \forall t\in[0,T].
\end{align}
{In view of Corollary~3.10 in \cite{BWY}, the optimal control of problem \eqref{LQ_cost} subjecting to \eqref{LQ} and \eqref{LQ_con} exists, and hence the constrained BSDE \eqref{LQ_BSDE} admits a solution.} We also note that, if one considers problem \eqref{LQ}-\eqref{LQ_cost} without constraint~\eqref{LQ_con}, the multiplier $\eta\in\Lb_{\Fb,\R+}^2$ becomes zero. Then, the condition~\eqref{minimal} trivially holds. Thus, the BSDE given in \thmref{SMP_con} becomes the classical one that
\begin{align}\label{LQ_BSDE_without_constraints}
\d Y_t=-\{b_1Y_t+b_2\E[Y_t]+\sigma_1Z_t+\sigma_2\E[Z_t]+q\widehat{X}_t\}\d t+Z_t\d W_t,~ Y_T=\ell\{\widehat{X}_T-\E[\widehat{X}_T]\}.
\end{align}
The stochastic minimal condition is given by, $\d t\times\d\Pb$-a.s.
\begin{align*}
\{b_3Y_t+b_4\E[Y_t]+\sigma_3Z_t+\sigma_4\E[Z_t]+v\widehat{\alpha}_t\}(u-\widehat{\alpha}_t)\geq 0,\quad \forall u\in U=\R.   
\end{align*}
This yields that the optimal control without constraint is   
\begin{align}\label{LQ_SMP_without_constraint}
\widehat{\alpha}_t=-v^{-1}\left\{b_3Y_t+b_4\E[Y_t]+\sigma_3Z_t+\sigma_4\E[Z_t]\right\},\quad \forall t\in[0,T],
\end{align}
which coincides with the result for the extended MFC problem in the LQ framework as studied in \cite{BWWX24} and an analytic solution can be derived by assuming that $Y_t$ takes the form $\beta_t X_t+\zeta_t\E[X_t]$ and solving the associated Riccati equations as in  \cite{BWWX24}.\\

\noindent\textbf{Acknowledgements} We sincerely thank the anonymous referee for helpful comments and suggestions, which improved the presentation of the paper. L. Bo and J. Wang are supported by National Natural Science Foundation of China (No. 12471451), Natural Science Basic Research Program of Shaanxi (No. 2023-JC-JQ-05), Shaanxi Fundamental Science Research Project for Mathematics and Physics (No. 23JSZ010). X. Yu is supported by the Hong Kong RGC General Research Fund (GRF) under grant no. 15211524 and grant no. 15214125.

\appendix
\section{Proofs of Auxiliary Results}\label{appendix}

\small
In this appendix, we collect proofs of some auxiliary results in the main body of the paper.

{\begin{proof}[Proof of Lemma \ref{interior}]
First,  let us assume that $y\in K_i^{\rm o}$ and there exists a $\xi\in K_i^+$ such that $\langle y,\xi\rangle_{\tY_i,\tY_i^*}=0$. As $y\in K_i^{\rm o}$, there exists a constant $\delta>0$ such that $B(y;\delta)\subset K_i^{\rm o}$. Then, for any $\epsilon\in (0,\delta)$ and any $v\in Y_i$, according to the definition of $K_i^+$, we conclude that $\langle y+\epsilon v,\xi\rangle_{\tY_i,\tY_i^*}\geq 0$, and hence $\langle v,\xi\rangle_{\tY_i,\tY_i^*}\geq 0$, which implies $\xi=0$.

Next, we argue by contradiction to show that the other side of the claim holds. Suppose that $\langle y,\xi\rangle_{\tY_i,\tY_i^*}>0$ for all nonzero $\xi\in K_i^+$ and $y\notin K_i^{\rm o}$, which implies $y\in \overline{K_i^c}$ with $K_i^c$ being the complement of $K_i$. Consequently, there exists a sequence $(y_k)_{k\in\mathbb{N}}\subset K_i^c$ such that $y_k\to y$ as $k\to \infty$. Then, in light of \lemref{dual}, there exists some $\xi_k\in K_i^+$ such that $\langle y_k,\xi_k\rangle_{\tY_i,\tY_i^*}<0$ for all $k\in\mathbb{N}$. We assume without loss of generality that $\|\xi_k\|_{\tY_i^*}=1$. Thus, by the separability, there exists a weak-$^*$ convergent subsequence of $(\xi_k)_{k\in\mathbb{N}}$ (still denoted by $(\xi_k)_{k\in\mathbb{N}}$) such that $\xi_k\to\xi\in \tY_i^*$ in the topology of weak-$^*$ convergence, and $\xi\in K_i^+$  results from the definition of $K_i^+$. It
then holds that
\begin{align*}
\left|\langle y_k,\xi_k\rangle_{\tY_i,\tY_i^*}-\langle y,\xi\rangle_{\tY_i,\tY_i^*}\right|&\leq\left|\langle y_k-y,\xi_k\rangle_{\tY_i,\tY_i^*}\right|+\left|\langle y,\xi_k-\xi\rangle_{\tY_i,\tY_i^*}\right|\leq\|y_k-y\|_{\tY_i}+\left|\langle y,\xi_k-\xi\rangle_{\tY_i,\tY_i^*}\right|.   
\end{align*}
Her, the second term converges to zero as $k\to\infty$ in lieu of the definition of weak-$^*$ convergence. Thus, $\langle y,\xi\rangle_{\tY_i,\tY_i^*}=\lim_{k\to\infty}\langle y_k,\xi_k\rangle_{\tY_i,\tY_i^*}\leq 0$, which contradicts $\langle y,\xi\rangle_{\tY_i,\tY_i^*}>0$, and the claim holds.
\end{proof}}

\begin{proof}[Proof of Corollary~\ref{CQcon}]
We first assume that MFCQ($\widehat{x}$) holds. If \corref{CQcon}-(iii) does not hold, then $r_0=0$, we can deduce that there exists a $\xi\in N_C(\widehat{x})$ such that\begin{equation}\label{contradiction}
-\xi=\sum_{i\in I}\lambda_i\circ Dg_i(\widehat{x})-\sum_{j\in J} \mu_j\circ Dh_j(\widehat{x}),
\end{equation} 
where $(\lambda_i)_{i\in I}$ and $(\mu_j)_{j\in J}$ can not be zero simultaneously thanks to \thmref{fj_condition}-(iv). By the definition of normal cone $N_C(\widehat{x})$, MFCQ($\widehat{x}$) and \lemref{interior}, the equality \equref{contradiction} can be written as
\begin{align*}
0\leq\langle c-\widehat{x},\xi\rangle=\sum_{i\in I}\langle Dg_i(\widehat{x})(c-\widehat{x}),\lambda_i\rangle_{\tY_i,\tY_i^*}-\sum_{j\in J}\langle Dh_j(\widehat{x})(c-\widehat{x}),\mu_j\rangle_{\tH_j}\leq 0.
\end{align*}
In above, the equality can only be attained when $\lambda_i=0$ for $i\in I$. Thus, the equality  \equref{contradiction} can be  translated into $\sum_{j\in J}\mu_j\circ Dh_j(\widehat{x})=\xi\in N_C(\widehat{x})$. Here, $\xi=(\xi_1,\xi_2)\in \tX_1^*\times \tX_2^*$ as the dual space of a product space is the product of their dual spaces. Hence, we have $\sum_{j\in J}\mu_j\circ D_{x_1}h_j(\widehat{x})=\xi_1$. Moreover, by the definition of normal cone, we can conclude that by choosing $c=(c_1,\widehat{x}_2)\in C\subset\tX_1\times\tX_2$,  where  we recall $\widehat{x}=(\widehat{x}_1,\widehat{x}_2)\in C\subset\tX$, we have   
\begin{align}\label{projection}
0\leq \langle \widehat{x}-c,\xi\rangle_{\tX,\tX^*}=\langle\widehat{x}_1-c_1,\xi_1\rangle_{\tX_1,\tX_1^*}.
\end{align}
As $\widehat{x}_1=\Pi_1(\widehat{x})\in(\Pi_1(C))^{\rm o}$, we deduce that \begin{equation}\label{mupro}
\sum_{j\in J}\mu_j\circ D_{x_1}h_j(\widehat{x}_1)=\xi_1=0.
\end{equation}
Due to the fact that $ (D_{x_1}h_j(\widehat{x}_1))_{j\in J}:\tX_1\to\prod_{j\in J}\tH_j$ is surjective, we get that $\mu_j=0$ $(j\in J)$, which contradicts the assertion in \lemref{fj_condition}-(iii). Thus, \corref{CQcon}-(iii) holds.

Now, let us assume that LICQ($\widehat{x}$) holds. From \equref{contradiction} and \equref{projection} again, it follows that
\begin{align*}
\sum_{i\in I}\lambda_i\circ D_{x_1}g_i(\widehat{x}_1)-\sum_{j\in J} \mu_j\circ D_{x_1}h_j(\widehat{x}_1)=-\xi_1=0.
\end{align*}
We argue by contradiction to show that $\lambda_i=0$ for $i\in I$. To this end, we first suppose that $\ker(\lambda_1)^c\neq\varnothing$. Then, by the surjection in LICQ($\widehat{x}$), there is an $x_0\in\tX_1$ s.t. $(D_{x_1}g_i(\widehat{x}_1)(x_0),D_{x_1}h_j(\widehat{x}_1)(x_1))_{i\in I,j\in J}$ belongs to the nonempty set $\ker(\lambda_1)^c\times\prod_{i\in I\backslash \{1\}}\ker(\lambda_i)\times\prod_{j\in J}\ker(\mu_j)$,  which yields that
\begin{align*}
0=\sum_{i\in I}\lambda_i\circ D_{x_1}g_i(\widehat{x}_1)(x_0)-\sum_{j\in J} \mu_j\circ D_{x_1}h_j(\widehat{x}_1)(x_0)=\left\langle D_{x_1}g_1(\widehat{x}_1)(x_0),\lambda_1\right\rangle_{\tY_1,\tY_1^*}\neq 0.
\end{align*}
This contradiction leads to $\ker(\lambda_1)=\tY_1$,  and hence $\lambda_1=0$. We then arrive at
\begin{align*}
\sum_{i\in I\backslash\{1\}}\lambda_i\circ D_{x_1}g_i(\widehat{x}_1)-\sum_{j\in J} \mu_j\circ D_{x_1}h_j(\widehat{x}_1)=0.
\end{align*}
Repeating the procedures above, we conclude that $\lambda_i=0$ for all $i\in I$. Hence, the above equality reduces to \equref{mupro}. Following the proof in the MFCQ case, we can deduce that $\mu_j=0$ for $j\in J$ as desired.  
\end{proof} 

{\begin{proof}[Proof of Proposition \ref{fj_sufficient}]
Recall $\xi\in N_C(\widehat{x})$ in the condition (iv) of \thmref{fj_condition}. From the convexity of $f$ on $C$, the condition (iv) of \thmref{fj_condition} (i.e., \eqref{FJ}) and $r_0>0$, it follows that, for any feasible point $x\in C$,
\begin{align*}
& r_0\left(f(x)-f(\widehat{x})\right)\geq \left\langle x-\widehat x,r_0Df(\widehat{x})\right\rangle_{\mathtt{X},\mathtt{X}^*}=\left\langle x-\widehat x,-\xi-\sum_{i\in I}\lambda_i\circ Dg_i(\widehat{x})+\sum_{j\in J}\mu_j\circ Dh_j(\widehat{x})\right\rangle_{\mathtt{X},\mathtt{X}^*}\\
&\qquad=\left\langle \widehat{x}-x,\xi\right\rangle_{\mathtt{X},\mathtt{X}^*}-\sum_{i\in I}\left\langle Dg_i(\widehat x)(x-\widehat{x}),\lambda_i\right\rangle_{\mathtt{Y},\mathtt{Y}^*}+\sum_{j\in J}\left\langle Dh_j(\widehat{x})(x-\widehat{x}),\mu_j\right\rangle_{\mathtt{H}_j}\\
&\qquad=\left\langle \widehat{x}-x,\xi\right\rangle_{\mathtt{X},\mathtt{X}^*}+\sum_{j\in J}\left\langle Dh_j(\widehat{x})(x-\widehat{x}),\mu_j\right\rangle_{\mathtt{H}_j}+\sum_{i\in I}\left\langle g_i(x)-g_i(\widehat{x})-Dg_i(\widehat{x})(x-\widehat{x}),\lambda_i\right\rangle_{\mathtt{Y},\mathtt{Y}^*}\\
&\qquad\quad-\sum_{i\in I}\left\langle g_i(x)-g_i(\widehat{x}),\lambda_i\right\rangle_{\mathtt{Y},\mathtt{Y}^*}.
\end{align*}
In light of the convexity of $g_i$ on $C$ for $i\in I$, we get that $g_i(x)-g_i(\widehat{x})-Dg_i(\widehat{x})(x-\widehat{x})\geq_{K_i}0$. This, together with $  \left\langle g_i(x)-g_i(\widehat{x}),\lambda_i\right\rangle_{\mathtt{Y},\mathtt{Y}^*}=\left\langle g_i(x),\lambda_i\right\rangle_{\mathtt{Y},\mathtt{Y}^*}\leq 0$, yields that 
\begin{align*}
\sum_{i\in I}\left\langle g_i(x)-g_i(\widehat{x})-Dg_i(\widehat{x})(x-\widehat{x}),\lambda_i\right\rangle_{\mathtt{Y},\mathtt{Y}^*}-\sum_{i\in I}\left\langle g_i(x)-g_i(\widehat{x}),\lambda_i\right\rangle_{\mathtt{Y},\mathtt{Y}^*}\geq 0.
\end{align*}
On the other hand, thanks to the affine representation of the equality constraint function $h_j$ for $j\in J$, we have that, for any feasible point $x\in C$, 
\begin{align*}
Dh_j(\widehat{x})(x-\widehat{x})=L_j(x-\widehat{x})=h_j(x)-h_i(\widehat{x})=0.
\end{align*}
As a result, by the construction of $\xi\in N_C(\widehat{x})$ (c.f. \eqref{eq:NCX}), we can deduce that, for any feasible point $x\in C$,
\begin{align}\label{eq:fxhatx0}
r_0(f(x)-f(\widehat{x}))\geq \langle \widehat{x}-x,\xi\rangle_{\mathtt{X},\mathtt{X}^*}\geq 0.
\end{align}
Then, the desired result follows from \eqref{eq:fxhatx0} and the fact $r_0>0$.
\end{proof}}

\begin{proof}[Proof of \lemref{continuous_ell}]
Let $i\in I_1$. Then, for any $t_1,t_2\in[0,T]$ with $t_1<t_2$, it holds that
\begin{align*}
|g_i(t_2)-g_i(t_1)|&\leq\E[|\phi^i(t_1,X_{t_1},\Law(X_{t_1}),\widetilde{\alpha}^{t_1})-\phi^i(t_2,X_{t_2},\Law(X_{t_2}),\widetilde{\alpha}^{t_2})|]\\
&\leq\E[|\phi^i(t_1,X_{t_1},\Law(X_{t_1}),\widetilde{\alpha}^{t_1})-\phi^i(t_2,X_{t_2},\Law(X_{t_2}),\widetilde{\alpha}^{t_1})|]\\
&\quad+\E[|\phi^i(t_2,X_{t_2},\Law(X_{t_2}),\widetilde{\alpha}^{t_1})-\phi^i(t_2,X_{t_2},\Law(X_{t_2}),\widetilde{\alpha}^{t_2})|]:=v_1+v_2.
\end{align*}
Note that $v_1$ converges to $0$ as $|t_1-t_2|\to 0$ by DCT and the continuity of $\phi^i$ in $(t,x,\mu)$. Furthermore, by DCT, $\E[\|\widetilde{\alpha}^{t_2}-\widetilde{\alpha}^{t_1}\|_{\X}^2]=\E[\int_{t_1}^{t_2}|\alpha_s|^2\d s]$ converges to $0$ as $|t_1-t_2|\to 0$ since $\alpha\in\Lb_{\Fb}^2$. For the term $v_2$, we have that, as $|t_1-t_2|\to 0$,
\begin{align*}
v_2&=\E[|\langle D_{\gamma}\phi^i(t_2,X_{t_2},\Law(X_{t_2}),\widetilde{\alpha}^{t_1}),\widetilde{\alpha}^{t_2}-\widetilde{\alpha}^{t_1}\rangle_{\X}|]+o(\E[\|\widetilde{\alpha}^{t_2}-\widetilde{\alpha}^{t_1}]\|_{\X})\\
&\leq\left(\E[\| D_{\gamma}\phi^i(t_2,X_{t_2},\Law(X_{t_2}),\widetilde{\alpha}^{t_1})\|_{\X}^2]\right)^{\frac12}\left( \E[\|\widetilde{\alpha}^{t_2}-\widetilde{\alpha}^{t_1}\|_{\X}^2]\right)^{\frac12}+o(\left(\E[\|\widetilde{\alpha}^{t_2}-\widetilde{\alpha}^{t_1}\|_{\X}^2]\right)^{\frac12}).
\end{align*}
By using \assref{ass1}-(A2), $\E[\| D_{\gamma}\phi^i(t_2,X_{t_2},\Law(X_{t_2}),\widetilde{\alpha}^{t_1})\|_{\X}^2]<\infty$, and hence $v_2\to 0$ as $|t_1-t_2|\to 0$ due to the fact $\E[\|\widetilde{\alpha}^{t_2}-\widetilde{\alpha}^{t_1}\|_{\X}^2]=\E[\int_{t_1}^{t_2}|\alpha_s|^2\d s]$. 
\end{proof} 

\begin{proof}[Proof of \lemref{well-posedSDE}]
For any $Z\in\Lb_{\Fb}^2$, the classical SDE theory (c.f. Karatzas and Shreve~\cite{KaraShreve1998}) gives that the following SDE has a unique strong solution $Y\in\Lb_{\Fb}^2$ to
\begin{align*}
\d Y_t=B(t,\omega,Y_t,\overline{Z}_t)\d t+\Sigma(t,\omega,Y_t,\overline{Z}_t),\quad Y_0\in L^2((\Omega,\F_0,\Pb);\R^n),  
\end{align*}
where, for $t\in[0,T]$, $\overline{Z}_t$ stands for the r.v. $Z_t$ instead of the evaluation $Z_t(\omega)$ at $\omega\in\Omega$. Then, we have constructed a self-mapping $G:\Lb_{\Fb}^2\mapsto\Lb_{\Fb}^2$ by $G(Z)=Y$. Thus, the existence of solution reduces to showing that the mapping $G$ admits a fixed point. For $i=1,2$, let $Y^i=(Y_t^i)_{t\in[0,T]}$ be the corresponding solution associated to $Z^i\in\Lb_{\Fb}^2$. Let $M=M(T)>0$ be a generic constant that may differ from line to line. The standard moment estimation under the assumption \equref{Lipschitz} on coefficients $(B,\Sigma)$ results in $\E\left[|Y_t^1-Y_t^2|^2\right]\leq M\int_0^t\left\{\E\left[|Y_s^1-Y_s^2|^2\right]+\E\left[|Z_s^1-Z_s^2|^2\right]\right\}\d s$ for $t\in[0,T]$. Applying the Gronwall's lemma, we obtain $\E\left[|Y_t^1-Y_t^2|^2\right]\leq M\int_0^t\E\left[|Z_s^1-Z_s^2|^2\right]\d s$ for $t\in[0,T]$, 
which gives the continuity of $G:\Lb_{\Fb}^2\mapsto\Lb_{\Fb}^2$. 

Next, let $Y^0_t=Y_0$ for all $t\in[0,T]$, and define $Y_t^n$ $(n\in\mathbb{Z}_+)$ by induction as follows:
\begin{align}\label{iteration}
Y_t^n=Y_0+\int_0^TB\left(t,\omega,Y_t^n,\overline{Y_t^{n-1}}\right)\d t+\int_0^T\Sigma\left(t,\omega,Y_t^n,\overline{Y_t^{n-1}}\right)\d W_t.
\end{align}
Under Assumption \equref{Lipschitz}, it holds that, for all $n\in\mathbb{N}$,
\begin{align*}
\E\left[|Y_t^{n+1}-Y_t^n|^2\right]\leq M\int_0^t\E\left[|Y_t^{n}-Y_t^{n-1}|^2\right]\d s,~~\forall t\in[0,T].    
\end{align*}
This implies that $|Y^{n+1}-Y^n|_{\Lb_{\Fb}^2}^2=O(\frac{1}{n!})$ as $n\to \infty$. Thus, $(Y_n)_{n\in\mathbb{N}}$ is a Cauchy sequence in $\Lb_{\Fb}^2$, and hence there exists some $Y\in\Lb_{\Fb}^2$ such that $Y^n\to Y$, as $n\to\infty$, in $\Lb_{\Fb}^2$. Letting $n\to\infty$, we see that $Y\in\Lb_{\Fb}^2$ is a fixed point of the mapping $G:\Lb_{\Fb}^2\mapsto\Lb_{\Fb}^2$. By Gronwall's lemma again, we conclude the uniqueness of the solution to SDE \equref{selfSDE}.
\end{proof}

{}
\end{document}